\documentclass[a4paper,11pt]{article}
\usepackage[utf8]{inputenc}
\usepackage[T1]{fontenc}
\usepackage[english]{babel}
\usepackage{lmodern,microtype}
\usepackage{amsmath,amssymb,amsthm,mathtools}
\usepackage{enumitem}
\setlist[enumerate]{nolistsep}
\usepackage[ruled,vlined]{algorithm2e}
\usepackage{lastpage}
\usepackage{tikz}

\usepackage[margin=2.5cm,top=2cm,bottom=2cm]{geometry}
\usepackage{fancyhdr}
\newtheorem{theorem}{Theorem}[section]
\newtheorem{lemma}[theorem]{Lemma}

\newtheorem{claim}[theorem]{Claim}
\newtheorem{corollary}[theorem]{Corollary}
\newtheorem{fact}[theorem]{Fact}
\newtheorem{question}[theorem]{Question}
\newtheorem{problem}[theorem]{Problem}
\newtheorem{conjecture}[theorem]{Conjecture}
\theoremstyle{definition}
\newtheorem{definition}[theorem]{Definition}
\newcommand{\defi}[1]{%
  \marginpar{\textcolor{black}{\tiny#1}}%
            {\textcolor{red!80!black}{\emph{#1}}}}
\renewcommand{\defi}[1]{%
            {\textcolor{red!80!black}{\emph{#1}}}}

\newcommand{\aname}[2]{\newcommand{#1}{#2}}

\aname{\Alon}{Alon}             
\aname{\Bender}{Bender}         
\aname{\Bernoulli}{Bernoulli}   
\aname{\Cayley}{Cayley}         
\aname{\Chebyshev}{Chebyshev}   
\aname{\Chung}{Chung}           
\aname{\Haggkvist}{H\"aggkvist} 
\aname{\Hall}{Hall}             
\aname{\Havet}{Havet}           
\aname{\Janson}{Janson}         
\aname{\Kuehn}{K\"uhn}          
\aname{\McDiarmid}{Mc\,Diarmid} 
\aname{\Moon}{Moon}             
\aname{\Mycroft}{Mycroft}       
\aname{\Osthus}{Osthus}         
\aname{\Redei}{R\'edei}         
\aname{\Saks}{Saks}             
\aname{\Sahili}{El~Sahili}      
\aname{\Shapira}{Shapira}       
\aname{\Sos}{S\'os}             
\aname{\Sudakov}{Sudakov}       
\aname{\Sumner}{Sumner}         
\aname{\Szemeredi}{Szemer\'edi} 
\aname{\Treglown}{Treglown}     
\aname{\Thomason}{Thomason}     
\aname{\Thomasse}{Thomass\'e}   
\aname{\Vondrak}{Vondr\'ak}     
\aname{\Wormald}{Wormald}       

\newcommand{\rarr}{\rightarrow}
\newcommand{\larr}{\leftarrow}

\newcommand{\cals}{\ensuremath{\mathcal{S}}}
\newcommand{\ee}{{\ensuremath{\mathrm{e}}}}
\newcommand{\bbe}{{\ensuremath{\mathbb{E}}}}
\newcommand{\bbn}{\ensuremath{\mathbb{N}}}
\newcommand{\bbp}{{\ensuremath{\mathbb{P}}}}
\newcommand{\sm}{\setminus}
\newcommand{\ipp}{{\ensuremath{i+1}}}
\newcommand{\imm}{{\ensuremath{i-1}}}
\newcommand{\tauend}{{\ensuremath{\tau_{\mathrm{end}}}}}
\newcommand{\calb}{{\ensuremath{\mathcal{B}}}}
\newcommand{\calt}{{\ensuremath{\mathcal{T}}}}
\newcommand{\lln}{{\ensuremath{\log \log n}}}
\DeclareMathOperator{\dist}{dist}
\DeclareMathOperator{\Var}{Var}
\newcommand{\deq}{\coloneqq}
\newcommand{\eps}{\ensuremath{\varepsilon}}
\let\phi\varphi

\let\emptyset\varnothing
\newcommand{\littleo}{\ensuremath{\mathrm{o}}}
\newcommand{\vcher}[1]{\ensuremath{\widehat{#1}}}
\newcommand{\vstar}{{{\ensuremath{v}}}}
\newcommand{\ink}{{\ensuremath{i\in[k]}}}
\newcommand{\jnk}{{\ensuremath{j\in[k]}}}
\newcommand{\dplus}{{\ensuremath{d_\geq}}}

\newcommand{\CCT}{{cycle of~cluster tournaments}}
\newcommand{\CsCT}{{cycles of~cluster tournaments}}

\newcommand{\oldqed}{}
\def\endofClaim{\hfill\scalebox{.6}{$\Box$}}
\newenvironment{claimproof}[1][Proof]
{ \renewcommand{\oldqed}{\qedsymbol}
  \renewcommand{\qedsymbol}{\endofClaim}
  \begin{proof}[#1]}
{ \end{proof}
  \renewcommand{\qedsymbol}{\oldqed}}

\begin{document}

\title{Unavoidable trees in tournaments}
\author{Richard Mycroft\thanks{\texttt{r.mycroft@bham.ac.uk}. School of Mathematics, University of Birmingham, Birmingham, B15 2TT, United Kingdom. 
Research partially supported by EPSRC grant EP/M011771/1.}~~and~T\'assio Naia\thanks{\texttt{tnaia@member.fsf.org}. School of Mathematics, University of Birmingham, Birmingham, B15 2TT, United Kingdom. 
Research supported by CNPq (201114/2004-3).}}

\date{}
\maketitle

\begin{abstract}
An oriented tree~$T$ on~$n$ vertices
is unavoidable if every tournament on~$n$ vertices 
contains a~copy of~$T$. In this paper we give a sufficient condition 
for~$T$ to be unavoidable, and use this to prove that almost all 
labelled oriented trees are unavoidable, verifying a conjecture of 
\Bender\ and~\Wormald. 
We additionally prove that
every tournament on~$n + \littleo(n)$ vertices 
contains a copy of every oriented tree~$T$ on~$n$ vertices with polylogarithmic maximum 
degree, improving a result of~\Kuehn, 
\Mycroft\ and~\Osthus.
\end{abstract}

\section{Introduction}\label{s:intro}

An oriented graph~$H$ on~$n$ vertices is~\defi{unavoidable} if every 
tournament on~$n$ vertices contains a copy of~$H$; otherwise, we say 
that~$H$ is~\defi{avoidable}. In particular, if~$H$ contains a directed
cycle then~$H$ must be avoidable, since a transitive tournament 
contains no directed cycles and hence no copy of~$H$. It is
therefore natural to ask which oriented trees are unavoidable.
A~classical result of \Redei~\cite{redei1934}
states that every directed path is unavoidable. More recently,
\Thomason\ showed that all orientations of sufficiently long cycles
are unavoidable except for those which yield directed cycles~\cite{Thomason1986}.
In~particular this implies that all orientations of sufficiently long 
paths are unavoidable. \Havet\ and~\Thomasse~\cite{HavetThomasse2000b}
then gave a complete answer for paths:
with three exceptions, every orientation of a path is unavoidable 
(the exceptions are antidirected paths of length~3,~5 and~7, which
are not contained in the directed cycle of length~3, the regular 5-vertex 
tournament and the Paley tournament on 7~vertices respectively). 
Significant attention has also been focused on the unavoidability of claws 
(a claw is an oriented graph formed by identifying the initial vertices of a 
collection of vertex-disjoint directed paths). Indeed~Saks and~S\'os~\cite{SaksSos81:unavoidable} 
conjectured that every claw on~$n$ vertices with maximum degree at most~$n/2$ is unavoidable. 
Lu~\cite{Lu91:claws,Lu93:claws} gave a counterexample to this conjecture, but in the other direction 
showed that every claw with maximum degree at most~$3n/8$ is unavoidable.
Lu, Wang and Wong~\cite{lu98:claws} then extended these results by showing that every claw with 
maximum degree at most~$19n/50$ is unavoidable, but that there exist claws with maximum degree 
approaching~$11n/23$ which are avoidable. Finding the supremum of all~$c > 0$ for which every 
claw with maximum degree at most~$cn$ is unavoidable remains an open problem.

Some oriented trees are far from being unavoidable. For example, the outdirected star~$S$ 
on~$n$ vertices (whose edges are oriented from the central vertex to 
each of the~$n-1$ leaves) is not contained in a regular tournament
on~$2n-3$ vertices, since each vertex of the latter has only~$n-2$ outneighbours. 
That is, there exist tournaments with almost twice as many vertices as~$S$ which do not
contain a copy of~$S$. On the other 
hand,~\Bender\ and~\Wormald~\cite{bender1988random} proved that 
almost all oriented trees are `almost unavoidable', in the sense that they 
are contained in almost all tournaments on the same number of vertices.

\begin{theorem}[\Bender\ and \Wormald, {\cite[Theorem~4.4]{bender1988random}}]
  Let~$T$ be chosen uniformly at random from the set of all labelled oriented
  trees on~$n$ vertices, and let~$G$ be chosen uniformly at random from the set 
  of all labelled tournaments on~$n$ vertices. 
  Then asymptotically almost surely~$G$~contains a~copy of~$T$.
\end{theorem}

In the same paper~\Bender\ and~\Wormald\ conjectured that this is 
true for every tournament~$G$, or, in other words, that almost all labelled
oriented trees are unavoidable. The main result of this paper is to prove
this conjecture.

\begin{theorem}\label{t:most-unavoidable}
  Let~$T$ be chosen uniformly at random from the set
  of all labelled oriented trees on~$n$ vertices.
  Then asymptotically almost surely~$T$ is unavoidable.
\end{theorem}

The following definitions are crucial for the proof of 
Theorem~\ref{t:most-unavoidable}. We say that a subtree~$T'$ of
a tree~$T$ is~\defi{pendant} if~$T-T'$ is connected. Next, we define 
`nice' oriented trees, whose properties are useful for embedding in 
tournaments, as follows (see Figure~\ref{fi:nice}).

\begin{figure}[tb] 
  \begin{center}
    \includegraphics{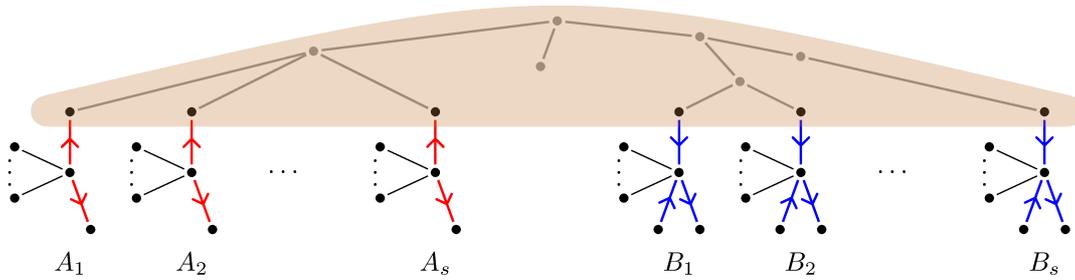}

    \caption{%
    An $\alpha$-nice tree~$T$ has $s = \lceil \alpha n \rceil$ pendant stars~$A_1, \dots A_s$ which contain an out-leaf of $T$ such that the edge between $T-A_i$ and $A_i$ is directed away from $A_i$, and also $s$ pendant stars~$B_1,\ldots,B_s$ which contain both an in-leaf of $T$ and an out-leaf of $T$ such that the edge between $T-B_i$ and $B_i$ is directed towards $B_i$. In this illustration we only indicate the orientations of edges specified by this definition. The shaded area is the subtree~$T -\bigcup_{i\in[s]} \bigl(V(A_i)\cup V(B_i)\bigr)$.}
    \label{fi:nice}
  \end{center}
\end{figure} 

\begin{definition}\label{d:nice-trees}
For~$\alpha > 0$ we say that an~oriented tree~$T$ on~$n$ vertices is~\defi{$\alpha$-nice}
if, writing~$s \deq \lceil\alpha n \rceil$, 
$T$~contains~$2s$ vertex-disjoint pendant oriented stars~$A_1,\ldots,A_s$ 
and~$B_1,\ldots,B_s$ 
such that  for each~\ink
\begin{enumerate}[label=(\roman*)]
\item $A_i$ is a subtree of~$T$ which contains an out-leaf of~$T$
  and the edge between~$A_i$ and~$T -A_i$ is oriented away from~$A_i$, and 
\item $B_i$ is a subtree of~$T$ which contains both an in-leaf of~$T$ and 
  an out-leaf of~$T$
  and the edge between~$B_i$ and~$T -B_i$ is oriented towards~$B_i$.
\end{enumerate}
\end{definition}

Most of the work involved in proving Theorem~\ref{t:most-unavoidable}
is in the proof of the following theorem,
which states that large nice oriented trees with polylogarithmic
maximum degree are unavoidable.

\begin{theorem}\label{t:good-unavoidable}
  For every~$\alpha,C > 0$ 
  there exists~$n_0$
  such that if~$T$ is an oriented tree on~$n \geq n_0$ vertices such that
\begin{enumerate}[label=(\roman*)]
\item $\Delta(T) \leq (\log n)^C$ and\label{i:log-deg}
\item $T$ is $\alpha$-nice, \label{i:nice}
\end{enumerate}
then~$T$ is unavoidable.
\end{theorem}

Almost all labelled trees satisfy condition~\ref{i:log-deg} of Theorem~\ref{t:good-unavoidable}, 
as proved by \Moon.

\begin{theorem}[{\cite[Corollaries~1 and~2]{moon70:trees}}]\label{t:aas-max-degree}
  For every~$\eps > 0$, if~$T$ is chosen uniformly at random from the set of all labelled 
  trees on~$n$ vertices, then
  asymptotically almost surely
  \[
  (1-\eps)\frac{\log n}{\log\log n}
  \leq\Delta(T)
  \leq (1+\eps)\frac{\log n}{\log\log n}.
  \]
\end{theorem}

Since a uniformly-random orientation of a uniformly-random labelled tree yields a uniformly-random labelled oriented tree, Theorem~\ref{t:aas-max-degree} remains valid if we replace `labelled tree' by `labelled oriented tree'.
We prove that almost all labelled oriented trees satisfy condition~\ref{i:nice} 
of~Theorem~\ref{t:good-unavoidable}.

\begin{theorem}\label{t:most-are-nice}
  Let~$T$ be chosen uniformly at random from the set
  of all labelled oriented trees on~$n$ vertices.
  Then asymptotically almost surely~$T$ 
  is~$\frac{1}{250}$-nice.
\end{theorem}

Combining Theorems~\ref{t:good-unavoidable},~\ref{t:aas-max-degree} and~\ref{t:most-are-nice} 
(with~$C=\eps = 1$ and~$\alpha = \frac{1}{250}$) immediately proves Theorem~\ref{t:most-unavoidable}. 

Another natural question is to find, for a given oriented tree~$T$,
the smallest integer~$g(T)$ such that every tournament on~$g(T)$ vertices contains a copy of~$T$.
In particular,~$T$ is~unavoidable if and only if~$g(T) = |T|$.
Sumner conjectured that for every oriented tree~$T$ on~$n$ vertices we have~$g(T) \leq 2n-2$,
and the example of an outdirected star described above demonstrates that this bound would be best possible.
\Kuehn, \Mycroft\ and~\Osthus~\cite{KMO10:sumner_exact, KMO11:sumner_approximate}
used a randomised embedding algorithm to~prove that Sumner's conjecture holds
for every sufficiently large~$n$;
previous upper bounds on~$g(T)$ had been established by~\Chung~\cite{Chung82:trees},
\Wormald~\cite{wormald83}, \Haggkvist\ and~\Thomason~\cite{HagThom91}, 
\Havet~\cite{Havet2002}, \Havet\ and~\Thomasse~\cite{havetThomasse2000} and \Sahili~\cite{Sahili2004}.
In particular, \Sahili\ proved that~$g(T) \leq 3n-3$ for every oriented tree~$T$ on~$n$ vertices,
 and this remains the best known upper bound on~$g(T)$ for small~$n$. 
\Kuehn, \Mycroft\ and~\Osthus~\cite{KMO11:sumner_approximate} also gave a stronger bound 
for large oriented trees of bounded maximum degree, proving that for every~$\alpha, \Delta > 0$, 
if~$n$ is sufficiently large then every oriented tree~$T$ on~$n$ vertices with~$\Delta(T) \leq \Delta$ 
has~$g(T) \leq (1+\alpha) n$. In other words, bounded degree oriented trees are close to being unavoidable, 
in that they are contained in every tournament of slightly larger order.

Our proof of Theorem~\ref{t:good-unavoidable} makes use of the aforementioned random embedding algorithm 
of~\Kuehn, \Mycroft\ and~\Osthus, using somewhat sharper estimates on certain quantities associated with 
the random embedding. In particular, using these stronger estimates we are able to establish the same 
bound on~$g(T)$ for oriented trees whose maximum degree is at most polylogarithmic in~$n$ (rather than
bounded by a constant as above). This is the following theorem, which we use repeatedly in the 
proof of Theorem~\ref{t:good-unavoidable}, and which may be of independent interest.

\begin{theorem}\label{t:log-bounded-deg}
For every~$\alpha,C>0$ there exists~$n_0$ such that if~$T$ is an oriented tree on~$n \geq n_0$ 
vertices with~$\Delta(T)\leq (\log n)^C$ and~$G$ is a~tournament on at least~$(1+\alpha) n$ 
vertices, then~$G$ contains a copy of~$T$.
\end{theorem}

\subsection{Proof outline for Theorem~\ref{t:good-unavoidable}}

Our proof of Theorem~\ref{t:good-unavoidable} uses a structural 
characterisation of large tournaments (Lemma~\ref{l:struct}) which is obtained by combining 
results of~\Kuehn, \Mycroft\ and~\Osthus~\cite{KMO11:sumner_approximate}. 
Loosely speaking, this shows that every 
large tournament~$G$ has one of the following two possible structures. 
The first possibility is that~$V(G)$ can be partitioned into two sets~$U$ and~$W$ such that 
almost all edges of~$G$ between~$U$ and~$W$ are directed from~$U$ to~$W$. 
We refer to such a structure as an `almost-directed pair'. 
The second possibility is that~$V(G)$ contains disjoint subsets~$V_1, \dots, V_k$ of equal size 
called `clusters' whose union includes almost all vertices of~$G$ and such that the edges 
of~$G$ directed from~$V_i$ to~$V_\ipp$ (with addition taken modulo~$k$) are `randomlike'. 
We refer to this structure as a `\CCT'. Given a 
tournament~$G$ on~$n$ vertices and a nice oriented tree~$T$ on~$n$ vertices 
with polylogarithmic maximum degree we consider separately these two cases for the 
structure of~$G$. 

\medskip \noindent {\bf Almost-directed pairs.}
Suppose first that~$G$ admits an almost-directed pair~$(U, W)$. 
In this case we begin by identifying the set~$B$ of `atypical' 
vertices of~$G$, namely those which lie too many edges directed `the wrong way', 
that is, from~$W$ to~$U$. Since~$(U, W)$ is an almost-directed pair~$B$ must be 
small. We then choose a set~$S$ of~$|B|$ distinct vertices of~$T$, each of which 
lies in an out-star of~$T$ and is adjacent to both an in-leaf and an out-leaf 
of~$T$. We also choose a small set~$S^-$ of vertices of~$T$, each of which lies 
in an in-star of~$T$ and is adjacent to an out-leaf of~$T$, and a small set~$S^+$ 
of vertices of~$T$, each of which lies in an out-star of~$T$ and is adjacent to an in-leaf of~$T$. 
The fact that~$T$ is nice ensures that we can choose such sets. 
Having done so, we form a subtree~$T'$ of~$T$ by removing one out-leaf adjacent 
to each vertex in~$S^-$, one in-leaf adjacent to each vertex in~$S^+$, and 
one in-leaf and one out-leaf adjacent to each vertex in~$S$. 
We then embed~$T'$ in~$G$; this can be achieved by \emph{ad hoc} 
methods (Lemma~\ref{l:approx-embed-in-cut}) using the fact 
that~$G$ has slightly more vertices than~$T'$ to give us a little `room to spare'. 
Moreover, we can insist that the image~$P^-$ of~$S^-$ under this embedding 
has~$P^- \subseteq U$, and likewise that the image~$P^+$ of~$S^+$ has~$P^+ \subseteq W$. 

It then suffices to embed the removed leaves into the set~$Q \subseteq V(G)$
of vertices of~$G$ not covered by the embedding of~$T'$. 
To do this, we first embed the removed leaves adjacent to vertices of~$S$ so as to cover 
the set~$B$ of atypical vertices of~$G$. This is achieved as follows. Let~$b$ be an 
atypical vertex of~$G$, choose a vertex~$s \in S$, and let~$s^+$ and~$s^-$ be the removed
out-leaf and in-leaf (respectively) adjacent to~$s$. Since~$s$ is a vertex of~$T'$, 
$s$ has already been embedded in~$G$, say to a vertex~$x$. Let~$x^+$ be an outneighbour 
of~$x$ in~$Q$, and let~$x^-$ be an inneighbour of~$x$ in~$Q$ (our embedding 
of~$T'$ in~$G$ will ensure that such vertices exist). Since~$G$ is a tournament, 
we must have either an edge~$b \rarr x$ or~$x \rarr b$ in~$G$. In the former case 
we embed~$s^-$ to~$b$ and~$s^+$ to~$x^+$, and in the latter case we 
embed~$s^+$ to~$b$ and~$s^-$ to~$x^-$; either way we have extended our 
embedding to cover the atypical vertex~$b$.

Having dealt with all atypical vertices in this manner, we let~$Q^- \subseteq U$ and~$Q^+ \subseteq W$ 
be the sets of vertices in~$U$ and~$W$ respectively which remain uncovered. The only vertices of~$T$ not 
yet embedded are the removed neighbours of vertices in~$S^- \cup S^+$. We now use the fact that all 
vertices of~$Q^-$ and~$Q^+$ are typical to find perfect matchings in the graphs~$G[P^- \rarr Q^+]$ 
and~$G[Q^- \rarr P^+]$ (our embedding of~$T'$ in~$G$ will ensure for this that we 
have~$|P^-| = |Q^+| = |P^+| = |Q^-|$). Recall that each~$s \in S^-$ was embedded to some 
vertex~$p \in P^-$, which is matched to some~$q \in Q^+$; we embed the removed outneighbour 
of~$s$ to~$q$. Likewise, each~$s \in S^+$ was embedded to some vertex~$p \in P^+$, which is 
matched to some~$q \in Q^-$; we embed the removed inneighbour of~$s$ to~$q$. This completes 
the embedding of~$T$ in~$G$.

\medskip \noindent {\bf Cycles of cluster tournaments.}
Now suppose that~$G$ contains an almost spanning \CCT\ with 
clusters~$V_1, \dots, V_k$ of equal size. Again we begin by identifying the small 
set~$B$ of atypical vertices, which in this case are those vertices in some 
cluster~$V_i$ which have atypically small inneighbourhood in~$V_{i-1}$ or atypically 
small outneighbourhood in~$V_{i+1}$, as well as those vertices not contained in any 
cluster~$V_i$. We also choose a small set~$L$ of vertices of~$T$ each of which is 
adjacent to at least one in-leaf and at least one out-leaf of~$T$ (this is possible 
since~$T$ is nice). Following this we split~$T$ into subtrees~$T_1$ and~$T_2$ 
which partition the edge-set of~$T$ and have precisely one vertex in common, so 
that~$T_1$ and~$T_2$ each contain many vertices of~$L$. Next we form subtrees~$T_1'$ 
and~$T_2'$ of~$T_1$ and~$T_2$ respectively by removing one in-leaf and one 
out-leaf adjacent to each vertex of~$L$. Finally, we embed~$T$ into~$G$ by the following two steps. 

First, we embed~$T_1$ in~$G$ so that all atypical vertices are covered and also so that the 
number of vertices of~$T_1$ embedded in each cluster~$V_i$ is approximately equal 
(more specifically, with an additive error on the order of~$\frac{n}{\lln}$). To do this, 
we apply a `random embedding algorithm' of~\Kuehn, \Mycroft\ 
and~\Osthus~\cite{KMO11:sumner_approximate} to embed~$T'_1$ into~$G$ so that 
approximately the same number of vertices of each cluster are covered and also 
so that roughly the same number of vertices of~$L$ are embedded to each cluster. 
(In fact, at this point we use slightly sharper estimates on the numbers of vertices 
embedded in each cluster than those given in~\cite{KMO11:sumner_approximate}; these 
arise from the same proofs). Then, by a similar argument to that used for covering 
atypical vertices in the previous case, for each~$i \in [k]$ and each vertex~$x \in L$ 
which was embedded in the cluster~$V_i$ we may use the fact that~$G[V_i]$ is a 
tournament to choose an atypical vertex~$b$ and an uncovered vertex~$y \in V_i$ so 
that the removed inneighbour and outneighbour of~$x$ can be embedded to~$b$ and~$y$. 
This gives the desired embedding of~$T_1$ in~$G$.

Secondly, to complete the embedding of~$T$ in~$G$ we embed~$T_2$ into the uncovered 
vertices of~$G$ (except for the single common vertex of~$T_1$ and~$T_2$ which is 
already embedded). For this we again apply the random embedding algorithm 
to embed~$T_2'$ in~$G$ with approximately the same number of vertices 
embedded within each cluster. We then carefully embed the removed inneighbours and 
outneighbours of a small number of vertices of~$L$ to achieve the following property. 
Let~$U_i \subseteq V_i$ be the set of vertices of~$V_i$ which remain uncovered, and 
let~$P_i \subseteq V_i$ be the image of vertices of~$L$ embedded to~$V_i$ whose 
removed inneighbour and outneighbour have not yet been embedded. We ensure 
that $2|P_1| = \dots = 2|P_k| = |U_1| = \dots = |U_k|$. Having done so, we partition 
each set~$U_i$ into two equal-size parts~$U_i^-$ and~$U_i^+$, and use the fact that 
all vertices which remain uncovered are typical to find perfect matchings 
in~$G[U_\imm^- \rarr P_i]$ and~$G[P_i \rarr U_\ipp^+]$ for each~$i \in [k]$. Then, 
for each vertex~$x$ in~$L$ whose removed inneighbour and outneighbour have not 
yet been embedded, let~$p \in P_i$ be the vertex to which~$x$ was embedded, and 
let~$q^-$ and~$q^+$ be the vertices to which~$p$ is matched in~$U_{i-1}$ and~$U_{i+1}$ 
respectively. We may then embed the removed inneighbour and outneighbour 
of~$x$ to~$q^-$ and~$q^+$ respectively; doing so for every~$x \in L$ completes the 
embedding of~$T$ in~$G$.

\subsection{Structure of this paper}
This paper is organised as follows. In Section~\ref{s:notation} we give definitions and 
preliminary results which we will use later on in the paper. These include structural 
results for tournaments and probabilistic estimates. 
Next, in Section~\ref{s:alloc-embed} we consider the `random embedding algorithm' 
of \Kuehn, \Mycroft\ and \Osthus~\cite{KMO11:sumner_approximate} and explain 
how to modify the proofs of the associated results to obtain 
slightly sharper bounds. In particular this includes Theorem~\ref{t:log-bounded-deg}; we also 
use these sharper bounds when considering \CsCT\ (as described in the 
proof sketch above). 
In Section~\ref{s:embed-in-cut} we consider tournaments~$G$ whose vertex set can be 
partitioned into two large sets which form an almost-directed pair in~$G$, proceeding 
as outlined in the proof sketch above to show that every such tournament contains a copy 
of every nice oriented tree of polylogarithmic maximum degree 
(this is Lemma~\ref{l:embed-nice-in-cut}, which can be interpreted 
as proving Theorem~\ref{t:good-unavoidable} for such tournaments). Then, in 
Section~\ref{s:embed-in-CCT} we do the same for tournaments~$G$ which contain an 
almost-spanning \CCT\ (Lemma~\ref{l:in-CCT}), making use of the sharper estimates 
established in Section~\ref{s:alloc-embed}.
In Section~\ref{s:main-thms} we prove Theorem~\ref{t:good-unavoidable} by using 
the structural results of Section~\ref{s:notation} 
to show that every tournament must have one of the two structures described above, 
and then applying the results of 
Sections~\ref{s:embed-in-cut} and~\ref{s:embed-in-CCT}.
We also give the proof of Theorem~\ref{t:most-are-nice}. 
Finally, in Section~\ref{s:conclusion} we conclude by discussing 
related results and possible areas for future research. 

\section{Notation and auxiliary results}\label{s:notation}

A \defi{directed graph}, or \defi{digraph} for short, consists of a vertex set~$V$ and 
edge set~$E$, where each edge is an ordered pair of distinct vertices. 
We think of the edge~$(u, v)$ as being directed from~$u$ to~$v$, and 
write~$x\rarr y$ or~$y\larr x$ to denote the edge~$(x, y)$. 
In a digraph~$G$, the \defi{outneighbourhood}~$N^+_G(x)$ of a~vertex~$x$ 
is the set~$\{\,y : x\rarr y\in E(G)\,\}$.
Similarly, the~\defi{inneighbourhood}~$N^-_G(x)$ of~$x$ is the set~$\{\,y : x\larr y\in E(G)\,\}$.
The~\defi{outdegree} and~\defi{indegree} of~$x$ in~$G$ 
are respectively~$\deg^+_G(x) \deq \bigl|N^+_G(x)\bigr|$ and~$\deg^-_G(x) \deq \bigl|N^-_G(x)\bigr|$, 
and the~\defi{semidegree}~$\deg^0_G(x)$ of~$x$ is the
minimum of the outdegree and indegree of~$x$.
The \defi{minimum semidegree} of~$G$ is~$\delta^0(G) \deq \min_{x\in V(G)} \deg^0_G(x)$.
For any subset~$Y \subseteq V(G)$,
we write~$\deg_G^-(x,Y)$ for~$|N_G^-(x)\cap Y|$, 
the \defi{indegree of~$x$ in~$Y$}; 
the \defi{outdegree of~$x$ in~$Y$},
denoted by~$\deg_G^+(x,Y)$, is~defined similarly. 
The \defi{semidegree of~$x$ in~$Y$},
denoted by~$\deg_G^0(x,Y)$,
is the~minimum of~those two~values.
We drop the subscript when there is no danger of confusion, 
writing~$N^-(x)$,~$\deg^0(x)$, and so forth. 
Also, we write~$|G|$ and~$e(G)$
for the number of vertices and edges of~$G$ respectively. 
For digraphs~$G$ and~$H$ we say that~$H$ is a \defi{subgraph} of 
$G$ if~$V(H) \subseteq V(G)$ and~$E(H) \subseteq E(G)$.
For any set~$X \subseteq V(G)$,
we write~$G[X]$ for the subgraph of~$G$ \defi{induced} by~$X$, which has
vertex set~$X$ and whose edges are all edges of~$G$ with both endvertices
in~$X$. If~$H$ is a subgraph of~$G$ then
we write~$G-H$ for~$G\bigl[V(G)\setminus V(H)\bigr]$. Likewise, for a vertex~$v$ or set of vertices~$S$, 
we write~$G - v$ or~$G - S$ for~$G\bigl[V(G)\setminus \{v\}\bigr]$ 
or~$G\bigl[V(G)\setminus S\bigr]$ respectively.

Every digraph considered in this paper will be an \defi{oriented graph}, 
meaning that there is at most one edge between each pair of vertices (and there are no loops). 
Equivalently, an oriented graph~$G$ can be formed by orienting each 
edge of some (undirected) graph~$H$; 
in this case we refer to~$H$ as the \defi{underlying graph} of~$G$, 
and say that~$G$ is an \defi{orientation} of~$H$. We refer to the~\defi{maximum degree} of 
an oriented graph~$G$, denoted~$\Delta(G)$, to mean the maximum degree of the underlying oriented graph~$H$.
A \defi{tournament} is an orientation of a~complete graph, 
and a \defi{subtournament} of a 
tournament~$G$ is a subgraph of~$G$ which is a tournament. 
A~\defi{regular tournament} is a tournament in which every vertex has equal indegree and outdegree; 
it is easily checked that regular tournaments of order~$n$ exist for every odd~$n \in \bbn$. 
A~\defi{transitive tournament} is a tournament whose vertices can be ordered~$v_1, \dots, v_n$ such 
that~$v_i \rarr v_j$ is an edge for each~$i < j$.
A~\defi{directed path} of length~$k$ is an oriented graph with vertex set~$v_0,\ldots, v_k$
and edges~$v_\imm\rarr v_i$ for each~$1 \leq i \leq k$, and an \defi{antidirected path} of 
length~$k$ is an oriented graph with vertex set~$v_0,\ldots, v_k$
and edges~$v_\imm\rarr v_i$ for odd~$i \leq k$ and~$v_\imm\larr v_i$ for even~$i \leq k$ 
(or \emph{vice versa}). A~\defi{directed cycle} of length~$k$ is an oriented graph with 
vertex set~$v_1,\ldots,v_k$ and edges
$v_i\rarr v_\ipp$ for each~$1 \leq i \leq k$ with addition taken modulo~$k$.

A~\defi{tree} is an acyclic connected graph, and an~\defi{oriented tree} 
or~\defi{directed tree} is an orientation of a tree. 
A~\defi{leaf} in a tree or oriented tree is a vertex of degree one. 
A~\defi{star} is a tree in which at most one~vertex (the~\defi{centre}) is not a leaf. 
A \defi{subtree}~$T'$ of a tree~$T$ is a subgraph of~$T$ which is also a tree, and we define subtrees 
of oriented trees similarly. For oriented trees~$T$ and~$T'$ we say that~$T'$ is 
an~\defi{out-subtree} (respectively an~\defi{in-subtree}) of~$T$ 
if both~$T'$ and~$T - T'$ are subtrees of~$T$, and
the unique edge of~$T$ between~$T'$ and~$T-T'$ is directed towards~$T'$
(respectively away from~$T'$). In a similar way we say that a vertex is 
an~\defi{in-leaf} or~\defi{out-leaf} of~$T$. 
Now let~$T$ be a tree or oriented tree.
It is often helpful to nominate a vertex~$r$ of~$T$ as the~\defi{root}
of~$T$; to emphasise this fact we sometimes refer to~$T$ as 
a~\defi{rooted tree}. If so, then every vertex~$x$ other than~$r$ has
a~unique~\defi{parent}; this is defined to be the neighbour~$p$ of~$x$
in the unique path in~$T$ from~$x$ to~$r$,
and~$x$ is said to be a~\defi{child} of~$p$. An~\defi{ancestral ordering} of the vertices of a rooted
tree~$T$ is an ordering of~$V(T)$ in which the root vertex appears first and every non-root vertex 
appears later than its parent.
Where it is clear from the context that an oriented tree is oriented, 
we may refer to it simply as a tree.

We say that a sequence of events~$A_1, A_2, \dots$ holds 
\defi{asymptotically almost surely}
if~$\bbp(A_n) \to 1$ as~$n \to \infty$.
Likewise, in this paper all occurrences of the standard asymptotic notation~$\littleo(f)$ 
refer to sequences~$f(n)$ with parameter~$n$ as~$n \to \infty$.
We will often have sets indexed by~$\{1,2,\ldots,k\}$ (\emph{e.g.}~$V_1, \dots, V_k$),
and addition of indices will always be~performed modulo~$k$.
Also, if~$\phi\colon A\to B$ is a function from~$A$ to~$B$ and~$A'\subseteq A$,
then we write~$\phi(A')$ for the image of~$A'$ under~$\phi$.
We omit floors and ceilings whenever they do not affect the argument,
and~write~$a = b \pm c$ to indicate that~$b-c\leq a\leq b+c$.
For~$k \in \bbn$ we denote by~$[k]$ the set~$\{1,2,\ldots,k\}$, and write~$\binom{S}{k}$
to denote the set of all~$k$-element subsets of a set~$S$.
We use the notation~$x\ll y$ to indicate that
for every positive~$y$ there exists a positive number~$x_0$
such that for every~$0<x<x_0$ the subsequent statements hold.
Such statements with more variables are defined similarly.
We always write~$\log x$ to mean the natural logarithm of~$x$.

\subsection{Structural results for tournaments}

Let~$G$ be a bipartite graph with vertex classes~$A$ and~$B$. Loosely speaking,
$G$ is~`regular' if 
the edges of~$G$ are `randomlike' in the sense that they 
are distributed roughly uniformly.
More formally, for any sets~$X \subseteq A$ and~$Y \subseteq B$, 
we write~$G[X,Y]$ for the bipartite subgraph of~$G$ with vertex classes~$X$ and~$Y$ 
and whose edges are the edges of~$G$
with one endvertex in each of the sets~$X$ and~$Y,$
and define the~\defi{density}~$d_G(X,Y)$ of edges between~$X$ and~$Y$ to be 
$d_G(X,Y) \deq e\bigl(G[X,Y]\bigr)/|X||Y|$.
Then, for any~$d, \eps > 0$, we say that~$G$ is~\defi{$(d,\eps)$-regular} if
for every~$X\subseteq A$ and every~$Y\subseteq B$
such that~$|X|\geq \eps|A|$ and~$|Y|\geq \eps|B|$
we have~$d_G(X,Y) = d \pm \eps$.
The following well-known proposition is immediate from this definition.

\begin{lemma}[Slicing lemma]\label{l:slice-pair}
  Fix~$\alpha, \eps, d > 0$ and let~$G$ be a~$(d,\eps)$-regular bipartite 
  graph with vertex classes~$A$ and~$B$. 
  If~$A'\subseteq A$ and~$B'\subseteq B$ have sizes 
  $|A'| \geq \alpha|A|$ and~$|B'| \geq \alpha|B|$,
  then~$G[A',B']$ is~$(d, \eps/\alpha)$-regular.
\end{lemma}

We say that~$G$ is~\defi{$(\dplus, \eps)$-regular} if~$G$ is~$(d', \eps)$-regular for
some~$d' \geq d$. Another immediate consequence of the definition of regularity is that, for small~$\eps$,
if~$G$ is~$(d, \eps)$-regular then almost all vertices of~$A$ have 
degree close to~$d|B|$ in~$B$ and almost all vertices of~$B$ have degree close
to~$d|A|$ in~$A$. We say that~$G$ is `super-regular' if no vertex has degree much lower 
than this. More precisely,~$G$ is~\defi{$(d,\eps)$-super-regular}
if~$(A, B)$ is~$(\dplus,\eps)$-regular and also for every~$a\in A$
and~$b\in B$ we have~$\deg(a,B) \geq (d-\eps)|B|$ and~$\deg(b,A) \geq (d-\eps)|A|$.

To complete the embedding of a spanning oriented tree in a tournament, we will make use of the 
following well-known lemma, which states that every balanced super-regular bipartite graph
contains a perfect matching (a bipartite graph is \defi{balanced} if its vertex classes have equal size).

\begin{lemma}\label{l:matching}
For every~$\eps, d > 0$ with~$d \geq 2\eps$, if~$G$ is a~$(d, \eps)$-super-regular balanced 
bipartite graph, then~$G$ contains a perfect matching.
\end{lemma}

\begin{proof}
Let~$A$ and~$B$ be the vertex classes of~$G$, and let~$m$ denote their common size.
  Consider an arbitrary set~$S \subseteq A$, and let~$N(S) \subseteq B$ denote the set of vertices 
  of~$G$ with a neighbour in~$S$.
  If~$|S|<\eps m$, then for each~$a \in S$ we have~$\deg(a, B) \geq (d-\eps)m \geq \eps m > |S|$, 
  so certainly~$\bigl|N(S)\bigr| \geq |S|$.
  Alternatively, if~$\eps m \leq |S|\leq (1-\eps)m$,
  then, since~$G$ is~$(\dplus, \eps)$-regular, at most~$\eps m$ vertices 
  of~$B$ have no neighbours in~$S$,
  so~$\bigl|N(S)\bigr| \geq (1-\eps)m\geq |S|$.
  Finally, if~$|S|>(1-\eps)m$
  then every vertex~$b \in B$ has a neighbour in~$S$,
  since~$\deg(b, A) \geq (d-\eps)m \geq \eps m > |A \sm S|$,
  so~$\bigl|N(S)\bigr| = m \geq |S|$. 
  In each case \Hall's criterion holds, that is, 
  we have~$\bigl|N(S)\bigr|\geq |S|$ for every subset~$S \subseteq A$.
  So~$G$ contains a perfect matching.
\end{proof}

Now let~$G$ be a digraph. For disjoint subsets~$X, Y \subseteq V(G)$ 
we denote by~$G[X \rarr Y]$, or equivalently by~\mbox{$G[Y\larr X]$}, 
the subdigraph of~$G$ with vertex set~$X \cup Y$ and edge set  
\[E(X \rarr Y) \deq \{\,x\rarr y \in E(G) : x\in X, y\in Y\,\}.\]
We call the ordered pair~$(X, Y)$ a \defi{directed pair} in~$G$ if there are no edges in~$G[X\larr Y]$,
that is, if every edge between~$X$ and~$Y$ is directed towards~$Y$.
Similarly, for any~$\mu \geq 0$ we call~$(X,Y)$ a~\defi{$\mu$-almost-directed pair} 
if~$e\bigl(G(X\larr Y)\bigr) \leq \mu|X||Y|$,
so any directed pair is a 0-almost-directed pair. These structures will play a key role in our proof.

Observe that the underlying graph of~$G[X \rarr Y]$
is a bipartite graph with vertex classes~$X$ and~$Y$. We say that~$G[X \rarr Y]$ is 
$(d, \eps)$-regular (respectively~$(d, \eps)$-super-regular) to mean that this underlying graph is
$(d, \eps)$-regular (respectively~$(d, \eps)$-super-regular). 
In this way we may apply the previous results of this subsection
to directed graphs. 

We now define another structure which is crucial for our proof.
Let~$d$ and~$\eps$ be positive real numbers, and let~$G$
be a digraph whose vertex
set is the disjoint union of~sets~$V_1,\ldots,V_k$.
We say that~$G$ is a~\defi{$(d,\eps)$-regular \CCT} 
if for each~\ink\ the induced subgraph~$G[V_i]$ is a tournament and the 
digraph~$G[V_i \rarr V_\ipp]$ is~$(\dplus,\eps)$-regular
(where addition on the subscript is taken modulo~$k$).
Likewise, we say that~$G$ is a~\defi{$(d, \eps)$-super-regular~\CCT}
if for each~\ink\ the induced subgraph~$G[V_i]$ is a tournament and the 
digraph~$G[V_i \rarr V_\ipp]$ is~$(d,\eps)$-super-regular.
In either case we refer to the sets~$V_1,\ldots,V_k$ as the~\defi{clusters}
of~$G$. 

The following lemma, a combination of two lemmas of \Kuehn, \Mycroft\
and~\Osthus~\cite{KMO11:sumner_approximate} about so-called `robust outexpanders', 
shows that every tournament with large minimum semidegree either admits a partition~$\{S, S'\}$ 
where~$S$ and~$S'$ are not too small and~$(S, S')$ is an almost-directed pair, or contains 
an almost-spanning \CCT.

\begin{lemma}[{\cite[Lemmas~2.7 and~2.8]{KMO11:sumner_approximate}}]\label{l:struct}
  Suppose that
  $1/n\ll 1/k_1\ll 1/k_0 \ll\eps\ll d\ll\mu\ll\nu\ll\eta,$
  and let~$G$ be a tournament on~$n$ vertices. Then either
\begin{enumerate}[label=(\alph*)]
\item $\delta^0(G) < \eta n$,
\item there is a partition of~$V(G)$ into sets~$S$ and~$S'$ with $\nu n < |S|,|S'|< (1-\nu)n$
and such that~$(S, S')$ is a~$\mu$-almost-directed pair in~$G$, or
\item there is an integer~$k$ with~$k_0\leq k\leq k_1$ for which~$G$ contains a~$(d,\eps)$-regular \CCT\ 
  with~clusters~$V_1,\ldots,V_k$ of equal size
  such that~$\bigl|\bigcup_{i=1}^k V_i\bigr|>(1-\eps)n$.
\end{enumerate}
\end{lemma}

\subsection{Useful estimates and bounds}

In this section we present various useful estimates.
The first is the following lemma 
which is used in Section~\ref{s:alloc-embed} to show that our
random allocation of vertices of an oriented tree~$T$ to the clusters of a cycle of cluster
tournaments~$G$ gives a~roughly uniform distribution.
We write~$\calb(n,p)$ to denote the binomial distribution
(the result of~$n$~independent \Bernoulli~experiments,
each with success probability~$p$).

\begin{lemma}\label{l:bin-modulo}
Suppose that~$1/n \ll 1/k$. If~$X \deq \calb\bigl(n,\frac{1}{2}\bigr)$, then for every~$r \in [k]$ we have
\[
\bbp(\,X\equiv r\mod k\,)=\frac{1}{k}\pm \frac{4}{\sqrt{n}}.
\]
\end{lemma}

\begin{proof}
Define~$p_\mu \deq \max_{x\in\{0,\ldots,n\}} \bbp(X=x)$. 
\Kuehn, \Mycroft\ and \Osthus~\cite[proof of Lemma~2.1]{KMO11:sumner_approximate} 
gave a straightforward argument to show that 
$\bbp(\,X\equiv r\mod k\,)=\frac{1}{k}\pm 2p_\mu$,
and the result then follows from a~standard estimate
on the binomial distribution (see, for example,~\cite[Section 1.2]{bollobas01:_random})
which states that~$p_\mu \sim 1/\sqrt{\pi n/2}$.
\end{proof}

The following straightforward lemma shows that a tournament can only have a few vertices of 
small in- or outdegree.

\begin{lemma}\label{l:semidegree}
For each~$d \in \bbn$, every tournament contains at most~$4d-2$ vertices 
with semidegree less than~$d$.
\end{lemma}

\begin{proof}
Let~$G$ be a tournament, and let~$X$ be the set of vertices~$x \in V(G)$ with 
$\deg^+(x) \leq d-1$. Then
\[\binom{|X|}{2} = e\bigl(G[X]\bigr) \leq \sum_{x \in X} \deg^+(x) \leq (d-1)|X|,\]
where the central inequality holds because every edge of~$G[X]$ contributes one
to the given sum. It follows that~$|X| \leq 2d-1$, that is,
there are at most~$2d-1$ vertices with outdegree less than~$d$. 
Essentially the same argument shows that there
are at most~$2d-1$ vertices with indegree less than~$d$, so in total at most 
$4d-2$ vertices have semidegree less than~$d$.
\end{proof}

Suppose~$N$ is an~$n$-element set, and let~$M$ be a subset of~$N$ with~$m$ elements.
If we choose a subset~$S \in \binom{N}{k}$ uniformly at random,
then the~random variable~$X = |S\cap M|$ is said to have \defi{hypergeometric distribution}
with parameters~$n,m$ and~$k$. Note that the expectation of~$X$ is then~$\bbe X = km/n$.

\begin{theorem}[\cite{janson00:_random_graph}, Corollary 2.3 and~Theorem~2.10]
  \label{t:exp}
  For every~$0<a<3/2$, if~$X$ has binomial or hypergeometric distribution,
  then~$\bbp\bigl(\,|X-\bbe X|\geq a\bbe X\,\bigr)\leq 2\exp(-a^2\bbe X/3)$.
\end{theorem}

We also use an Azuma-type concentration result
for martingales due to \McDiarmid~\cite{mcdiarmid98:_concen},
in the form stated by~\Sudakov\ and~\Vondrak~\cite{sudakov10}.

\begin{lemma}\label{l:martingale}
Fix~$n\in\bbn$
and let~$X_1,\ldots,X_n$ be random variables taking values in~$[0,1]$
such that for each~$i\in[n]$ we have
\(\bbe(\,X_i\mid X_1,\ldots,X_\imm\,)\leq a_i\).
If~$\mu\geq\sum_{i=1}^n a_i$, then for every~$\delta$ with~$0<\delta<1$ we have
\[
\bbp\Bigl(\,\sum_{i=1}^n X_i > (1+\delta)\mu\,\Bigr)\leq \ee^{-\delta^2\mu/3}.
\]
\end{lemma}

\section{Allocating and embedding}\label{s:alloc-embed}

In this section we show how to obtain somewhat sharper estimates from the random allocation and 
embedding algorithms used by \Kuehn, \Mycroft\ and~\Osthus~\cite{KMO11:sumner_approximate} to embed 
oriented trees in slightly larger tournaments. We begin with the following lemma, which is a 
slightly modified version of~\cite[Lemma~2.10]{KMO11:sumner_approximate}.

\begin{lemma}\label{l:F_i-stren}
  For every~$C>0$ there exists~$n_0$ such that for every rooted tree~$T$ on~$n \geq n_0$ vertices 
  with~$\Delta(T)\leq (\log n)^C$ and root~$r$, there exist~$s \in \bbn$, pairwise-disjoint 
  subsets~$F_1\ldots,F_s \subseteq V(T)$,
  and not-necessarily-distinct vertices~$v_1,\ldots,v_s$ of~$T$ with the following properties.
  \begin{enumerate}[label=(\arabic*)]
  \item \label{i:3.1i} $\bigl|\bigcup_{i\in[s]} F_i\bigr|\geq n-n^{5/12}$.\label{i:improved}
  \item \label{i:3.1ii} $|F_i|\leq n^{2/3}$ for each~$i\in[s]$.
  \item \label{i:3.1iii} For any~$i\in[s]$, any~$x\in\{r\}\cup \bigcup_{j<i}F_j$, and any~$y\in F_i$, 
  the path from~$x$ to~$y$ in~$T$ includes~$v_i$.
  \item \label{i:3.1iv} For any~$i \in [s]$ and~$y\in F_i$ we have~$\dist_T(v_i,y)\geq {(\log \log n)^3}$.
  \end{enumerate}
\end{lemma}

The original version of this lemma had constants~$\Delta, \eps, k > 0$ rather than~$C > 0$, 
assumed additionally that~$\Delta(T) \leq \Delta$, had~$n-\eps n$ in place of~$n-n^{5/12}$ in~\ref{i:3.1i} 
and had~$k$ in place of~$\lln$ in~\ref{i:3.1iv}. 
However, the form of the lemma given above can be established 
by an essentially identical proof, replacing each instance of~$k$ by~$\lln$ and each instance 
of~$\Delta$ by~$(\log n)^C$. The crucial point is that we then replace 
the bound~$3 n^{1/3}\Delta^{k^3} \leq \eps n$ by the bound $3n^{1/3}(\log n)^{C{(\log \log n)^3}} \leq n^{5/12}$. 
These changes yield~\ref{i:3.1i} and~\ref{i:3.1iv} above, 
whilst~\ref{i:3.1ii} and~\ref{i:3.1iii} are unchanged.

We now consider the random allocation algorithm of \Kuehn, \Mycroft\ 
and~\Osthus~\cite[Vertex Allocation Algorithm]{KMO11:sumner_approximate}, 
which is presented below as Algorithm~\ref{a:alloc}. Given a rooted oriented 
tree~$T$ and a~\CCT~$G$ with clusters~$V_1, V_2, \dots, V_k$, this assigns each vertex 
of~$T$ to a cluster of~$G$. 
We allocate vertices of~$T$ one at a~time in an~ancestral ordering. 
This ensures that whenever we allocate a vertex~$x$ other than the root, 
the parent~$p$ of~$x$ has previously been embedded to some cluster~$V_i$. 
We then say that~$x$ is allocated \defi{canonically} if either~$p\rarr x\in E(T)$ and~$x$ is allocated 
to the cluster~$V_\ipp$, or~$p \larr x\in E(T)$ and~$x$ is allocated to the cluster~$V_\imm$. 
Moreover, we say that an allocation of the vertices of~$T$ to the clusters of~$G$ is \defi{semi-canonical} 
if every vertex of~$T$ is either allocated canonically or allocated to the same cluster as its parent, every 
vertex adjacent to the root of~$T$ is allocated canonically, and for each~$\ink$ the set~$U_i$ of vertices 
allocated to~$V_i$ induces a forest~$F = T[U_i]$ in which no connected component has more than~$\Delta(T)$ 
vertices. 

\smallskip
\begin{algorithm}[H]
\SetAlgoVlined
\caption{The Vertex Allocation Algorithm~\cite{KMO11:sumner_approximate}}\label{a:alloc}
\DontPrintSemicolon
\SetKwInOut{Input}{Input}
\Input{an oriented tree~$T$ on~$n$ vertices, a root vertex~$t_1$ of~$T$, and clusters~$V_1,\ldots,V_k$.}
Choose an ancestral ordering~$t_1,\ldots,t_n$ of~$V(T)$.\\
\For{$\tau = 1$ \KwTo $n$}{%
  \lIf{$\tau=1$}{allocate~$t_1$ to~$V_1$.}
  \Else{%
    Let~$t_\sigma$ be the parent of~$t_\tau$.\;
    \lIf{$\dist_T(t_\tau,t_1)$ is odd}{allocate~$t_\tau$ canonically.}
    \lElse{%
    Allocate~$t_\tau$ to the same cluster
    as~$t_\sigma$ with probability~$1/2$ and\newline
    allocate~$t_\tau$ canonically with probability~$1/2$,
    independently of all previous choices.}
  }
}
\end{algorithm}
\smallskip

The following lemma, a slightly modified version of~\cite[Lemma~3.3]{KMO11:sumner_approximate}, 
states that Algorithm~\ref{a:alloc} will always return a semi-canonical allocation, and moreover 
that if~$T$ is sufficiently large then the allocation of vertices to clusters will be approximately uniform.

\begin{lemma}\label{l:3.3-stren}
  Let~$T$ be an oriented tree on~$n$ vertices rooted at~$r$.
  If we allocate the vertices of~$T$ to~clusters~$V_1,\ldots,V_k$
  by~applying the Vertex Allocation Algorithm,
  then the following properties hold.
  \begin{enumerate}[label=(\alph*)]
  \item\label{l:3.3-stren-i} 
    The allocation obtained will be semi-canonical.
  \item \label{l:3.3-stren-ii}
    Let~$u$ and~$v$ be distinct vertices of~$T$ such that~$u$ lies on the path from~$r$ to~$v$,
    let~$P$ be the path between~$u$ and~$v$,
    and let~$E \subseteq V(T)$ consist of all vertices~$x \in V(P) \sm \{u\}$
    for which~$\dist(r, x)$ is even. 
    If we condition on the event that~$u$ is allocated to some cluster~$V_j$,
    then~$v$ is allocated to cluster~$V_{j + R + F}$
    (taking addition in the subscript modulo~$k$) where~$R \deq\calb\bigl(|E|,\frac{1}{2}\bigr)$
    and~$F$ is a deterministic variable depending only on~$\dist(r, u)$
    and the orientations of edges of~$P$
    (that is,~$F$ is unaffected by the random choices made by the Vertex Allocation Algorithm).
  \item\label{l:3.3-stren-iii}
    Suppose that~$1/n \ll 1/k$. 
    Let~$u$ and~$v$ be vertices of~$T$ 
    such that~$u$ lies on the path from~$r$ to~$v$, 
    and~$\dist_T(u,v)\geq {(\log \log n)^3}$. 
    Then for any~$i,j\in[k]$,
    \[
    \bbp(\,\text{$v$ is allocated to~$V_i$} \mid \text{$u$ is allocated to~$V_j$}\,)
      =\frac{1}{k}\left(1 \pm \frac{1}{4 \lln}\right).
    \]
  \item\label{l:3.3-stren-iv}
    Suppose that~$1/n\ll 1/k,\alpha,1/C$ and
    that~$\Delta(T)\leq (\log n)^C$. Let~$S$ be a~subset of~$V(T)$ 
    with at least~$\alpha n$ vertices.
    Then with probability~$1-\littleo(1)$
    each of the~$k$~clusters~$V_i$ 
    has~$|S|\bigl(\frac{1}{k} \pm \frac{1}{\lln}\bigr)$ vertices of~$S$ allocated to it.
  \end{enumerate}
\end{lemma}

The statement above differs from the original version of the lemma in the following ways. 
Firstly,~\ref{l:3.3-stren-ii} was not stated explicitly, 
but was established in the original proof. Secondly, the original version of~\ref{l:3.3-stren-iii} instead
had constants~$1/k \ll \delta$, assumed that~$\dist_T(u,v)\geq k^3$ instead 
of~$\dist_T(u,v)\geq {(\log \log n)^3}$, and had~$\delta$ in place of~$\frac{1}{\lln}$ in the displayed equation.
Finally, the original version of~\ref{l:3.3-stren-iv} had constants~$1/n \ll 1/\Delta, 1/k \ll \delta$, 
assumed instead that~$\Delta(T) \leq \Delta$, had~$\delta$ in place of~$\frac{1}{\lln}$, was only 
stated for the special case~$S = V(T)$, and only provided an upper bound on the number of vertices allocated to each cluster.
So our version of the lemma allows the bounds in~\ref{l:3.3-stren-iii} and~\ref{l:3.3-stren-iv} 
to decrease with~$n$, and~$\Delta(T)$ to grow with~$n$, rather than being fixed constants. 
We now show how the original proof can be modified to establish our altered versions 
of~\ref{l:3.3-stren-iii} and~\ref{l:3.3-stren-iv}. 

\begin{proof}
To prove~\ref{l:3.3-stren-iii}, let~$\ell \deq \dist_T(u, v)$,
and define~$E$ as in~\ref{l:3.3-stren-ii}, so~$|E| = \lfloor \frac{\ell}{2} \rfloor$ 
or~$|E| = \lceil \frac{\ell}{2} \rceil$.
By~\ref{l:3.3-stren-ii} it suffices to show 
that $\bbp\bigl(\,\calb\bigl(|E|,\frac{1}{2}\bigr) = r \mod k\,\bigr) = \frac{1}{k} \pm \frac{1}{4k\lln}$ 
for each~$r \in [k]$, and since~$|E| \geq \frac{1}{3} {(\log \log n)^3}$ this holds by Lemma~\ref{l:bin-modulo}.

We now prove \ref{l:3.3-stren-iv}.
By Lemma~\ref{l:F_i-stren}, there exist an integer~$s \leq 3n^{1/3}$,
vertices~$v_1,\ldots,v_s\in V(T)$
and pairwise-disjoint subsets~$F_1,\ldots,F_s$ of~$V(T)$
such that~$|\bigcup_{i=1}^s F_i|\geq n-n^{5/12}$
and~$|F_i|\leq n^{2/3}$ for each~$\ink$, such that if~$j<i$,
then any path from~$r$ or any vertex of~$F_j$
to any vertex of~$F_i$ passes through the vertex~$v_i$,
and also such that~$\dist(v_i,F_i)\geq {(\log \log n)^3}$. Write~$\delta \deq \frac{1}{\lln}$; we shall prove that 
\[
(\dagger)\quad
\parbox{.7\textwidth}{with probability~$1-\littleo(1)$, for any~$j \in [k]$
the total number of vertices from~$\bigcup_{i\in[s]} F_i\cap S$
allocated to cluster~$V_j$ is at~most~$|S|(\frac{1}{k} + \frac{\delta}{2k})$.}
\]
Note that ($\dagger$)~implies~\ref{l:3.3-stren-iv}.
Indeed, since the number of vertices of~$T$
not contained in any of the sets~$F_i$
is at most~$n^{5/12} \leq \alpha n/2k\lln \leq \delta |S|/2k$, if ($\dagger$) holds
then for any~$\jnk$ in total at most  
$|S|(1+\delta)/k$~vertices 
of~$S$ are allocated to~$V_j$. It follows that at least 
$|S|-(k-1)|S|(1 +\delta)/k \geq |S|(1/k - \delta)$~vertices of~$S$
are allocated to~$V_j$, so~\ref{l:3.3-stren-iv} holds.

To prove~($\dagger$),
define random variables~$X_i^j$ for each~$i\in[s]$ and~$j\in[k]$ by
\[
X_i^j\deq\frac{\text{\# of vertices of~$F_i\cap S$ allocated to cluster~$V_j$}}{n^{2/3}},
\]
so each~$X_i^j$ lies in the range~$[0,1]$.
Then since the cluster to which a vertex~$x$ of~$T$ is allocated
is dependent only on the cluster to which the parent of~$x$ is allocated
and on the outcome of the random choice made when allocating~$x$,
we have for each~$q \in [k]$ 
that~$\bbe(\,X_i^j \mid X_\imm^j,\ldots,X_1^j, v_i\in V_q\,)=\bbe(\,X_i^j\mid v_i\in V_q\,)$, 
where we write~$x \in V_q$ to denote the event that~$x$ is allocated
to~$V_q$. So for any~$i \in [s]$ and~$\jnk$ we have
\begin{align*}
\bbe(\,X_i^j \mid X_\imm^j,\ldots,X_1^j\,)
& \leq
  \max_{q\in[k]} \bbe(\,X_i^j \mid X_\imm^j,\ldots,X_1^j, v_i\in V_q\,)
 =
  \max_{q\in[k]} \bbe(\,X_i^j \mid v_i\in V_q\,) \\
& =
  \max_{q\in[k]} \frac{\sum_{x\in F_i\cap S} \bbp(\,x\in V_j\mid v_i\in V_q\,)}{n^{2/3}}
\leq \left(\frac{1}{k}+\frac{\delta}{4k}\right)\frac{|F_i\cap S|}{n^{2/3}},
\end{align*}
using~\ref{l:3.3-stren-iii}.
We apply Lemma~\ref{l:martingale} with
\[
\mu
\deq \left(\frac{1}{k}+\frac{\delta}{4k}\right) \frac{|S|}{n^{2/3}}
\geq \left(\frac{1}{k}+\frac{\delta}{4k}\right) \sum_{i\in[s]}
\frac{|F_i\cap S|}{n^{2/3}},
\]
to obtain
\[
\bbp\Bigl(\sum_{i\in[s]} X_i^j > (1+\delta/8)\mu\Bigr)
\leq \exp\left(\frac{-(\delta/8)^2\mu}{3}\right)
= \exp \left(-\frac{\delta^2(1+\delta/4) |S|}{192kn^{2/3}}\right)
 \leq \exp\bigl(-n^{1/4}\bigr)
\]
where the~second inequality holds since we assumed that~$1/n\ll 1/k,\alpha$. Taking a union bound,
we find that with probability~$1-\littleo(1)$ we have for each~\jnk~that
\begin{align*}
n^{2/3}\sum_{i\in[s]} X_i^j
  \leq n^{2/3}(1+\delta/8)\mu  
  \leq |S|\left(\frac{1}{k} + \frac{\delta}{2k}\right).
\end{align*}
In other words, for each~\jnk\ there are 
at~most $|S|\bigl(\frac{1}{k}+ \frac{\delta}{2k}\bigr)$~vertices
from~$\bigcup_{i=1}^r F_i\cap S$ allocated to~$V_j$, so~($\dagger$) holds.
\end{proof}

Having applied the random allocation algorithm to allocate the vertices 
of an oriented tree~$T$ to the clusters of a slightly larger \CCT~$G$,
\Kuehn, \Mycroft\ and \Osthus\ proceeded to embed~$T$ in~$G$ using a
vertex embedding algorithm which successively embedded vertices of~$T$
in~$G$ following an ancestral ordering of the vertices of~$T$, with
each vertex being embedded in the cluster to which is was
allocated. Studying this algorithm yields the following lemma,
which is a modified form of~\cite[Lemma~3.4]{KMO11:sumner_approximate}.

\begin{lemma}
  \label{l:approx-embedding}
  Suppose that~$1/n \ll 1/C$ and that
  \(
  1/n \ll 1/k\ll \eps\ll\gamma\ll d\ll \alpha,
  \)
  and let~$m\deq n/k$.
  \begin{enumerate}[label=(\arabic*)]
  \item
    Let~$T$ be an oriented tree on at most~$n$ vertices with 
    root~$r$ and~$\Delta(T)\leq (\log n)^C$.
  \item \label{l:app-emb-b}
    Let~$G$ be a~$(d,\eps)$-regular \CCT\
    on clusters~$V_1,\ldots,V_k$, each of size at least~$(1+\alpha)m$ and at most~$3m$,
    and let~\vstar\ be a vertex of~$V_1$ with
    at least~$\gamma m$ inneighbours in~$V_k$ and
    at least~$\gamma m$ outneighbours in~$V_2$.
  \item
    Let the vertices of~$T$ be allocated to the clusters~$V_1,\ldots,V_k$
    so that at most~$(1+\alpha/2)m$ vertices 
    are allocated to each cluster~$V_i$,
    and so that the allocation is semi-canonical.
  \end{enumerate}
  Then~$G$ contains a copy of~$T$ in which~$r$ is embedded to~\vstar,
  and such that each vertex is~embedded
  in the cluster to which it was allocated.
\end{lemma}

The differences between Lemma~\ref{l:approx-embedding} as stated above 
and the original version in~\cite{KMO11:sumner_approximate} are twofold. 
Firstly, the original assumption that~$\Delta(T) \leq \Delta$ for some 
(fixed)~$\Delta$ with~$1/n \ll 1/\Delta \ll \eps$ has been replaced by our 
assumption that~$\Delta(T) \leq (\log n)^C$. Secondly, we allow the cluster sizes to 
vary between the bounds in~\ref{l:app-emb-b}, whereas the original form insisted that all
clusters have size exactly~$(1+\alpha)n$. Neither of these changes materially affects the 
original proof given in~\cite{KMO11:sumner_approximate}.

Combining Lemma~\ref{l:3.3-stren} and Lemma~\ref{l:approx-embedding} 
immediately yields the following corollary, a modified version 
of~\cite[Lemma~3.2]{KMO11:sumner_approximate}, in which the original constant 
bound on~$\Delta(T)$ has been replaced by a polylogarithmic bound.

\begin{corollary}
  \label{c:cctcombined}
  Suppose that~$1/n \ll 1/C$ and that $1/n \ll 1/k \ll \eps \ll d\ll \alpha \leq 2$,
  and let~$m\deq n/k$. Let~$T$ be an oriented tree on at 
  most~$n$ vertices with~$\Delta(T)\leq (\log n)^C$ and with root~$r$. 
  Also let~$G$ be a~$(d,\eps)$-regular \CCT\
  on clusters~$V_1,\ldots,V_k$, each of size~$(1+\alpha)m$, 
  and let~$v$ be a vertex of~$V_1$ with at least~$d^2 m$ inneighbours in~$V_k$ 
  and at least~$d^2 m$ outneighbours in~$V_2$.
  Then~$G$ contains a copy of~$T$ in which~$r$ is embedded to~$v$.
\end{corollary}

Recall that Theorem~\ref{t:log-bounded-deg} of this paper is a sharpened version 
of~\cite[Theorem~1.4(2)]{KMO11:sumner_approximate}. The proof of Theorem~\ref{t:log-bounded-deg} 
is identical to the proof of~\cite[Theorem~1.4(2)]{KMO11:sumner_approximate} given 
in~\cite{KMO11:sumner_approximate} from this point onwards, 
using Corollary~\ref{c:cctcombined} above in place of~\cite[Lemma~3.2]{KMO11:sumner_approximate}.

In the proof of Theorem~\ref{t:good-unavoidable} we use the following corollary. 
This is a consequence of Theorem~\ref{t:log-bounded-deg} and \Sahili's theorem~\cite{Sahili2004} that, 
for every~$m \in \bbn$, every tournament on at least~$3m-3$ vertices contains every oriented 
tree on~$m$ vertices. Indeed, this corollary is simpler to apply since it holds for both small 
and large trees.

\begin{corollary}\label{cr:S-KMO}
Suppose that~$1/n \ll \alpha,1/C$. 
Let~$T$ be an oriented tree on~$n' \leq n$ vertices with~$\Delta(T)\leq (\log n)^C$,
and let~$G$ be a~tournament on at least~$n' + \alpha n$ vertices. 
Then~$G$ contains a copy of~$T$.
\end{corollary} 

\begin{proof}
Fix~$\alpha, C > 0$ and choose~$n_0$ sufficiently large to apply Theorem~\ref{t:log-bounded-deg} 
with~$2C$ in place of~$C$, and also so that~$\log n_0 \geq \bigl(1 + \log (2/\alpha)\bigr)^2$.
Then we may assume that~$n \geq 2 n_0/\alpha$.
If~$n' > \alpha n/2$, then~$n' > n_0$,
so~$G$ contains a copy of~$T$ by Theorem~\ref{t:log-bounded-deg},
since~$G$ has at least~$n' + \alpha n \geq (1+\alpha)n'$ vertices 
and~$\Delta(T) \leq (\log n)^C \leq (\log n')^{2C}$.
On the other hand, if~$n' \leq \alpha n/2$, then~$|G| \geq n' + \alpha n \geq 3n'$, 
and thus~$G$ contains a copy of~$T$ by the aforementioned theorem of \Sahili.
\end{proof}

The modified proofs of Lemmas~\ref{l:F_i-stren} and~\ref{l:approx-embedding} and 
Theorem~\ref{t:log-bounded-deg} are presented in full in~\cite{naia18}.

\section{Almost-directed pairs}\label{s:embed-in-cut}

Our aim in this section is to prove Lemma~\ref{l:embed-nice-in-cut}, which states that
every nice oriented tree~$T$ of polylogarithmic maximum degree is contained in
every tournament whose vertex set admits a partition~$\{U, W\}$ into 
not-too-small sets~$U$ and~$W$ such that the pair~$(U, W)$ is almost-directed.

We begin with a definition and two lemmas.
If~$(X, Y)$ is a~$\mu$-almost-directed pair in a digraph~$G$, we say that an edge
$e \in E(G)$ is a~\defi{reverse edge} if~$e\in E(X\larr Y)$ (so, by definition, an almost-directed pair
has at most~$\mu |X||Y|$ reverse edges).
Our first lemma guarantees that we may partition the vertex set of an oriented tree~$T$ into 
sets~$A$ and~$B$ so that~$(A, B)$ is a directed pair in~$T$ and so that specific in-subtrees of~$T$
have all their vertices in~$A$
and specific out-subtrees of~$T$ have all their vertices in~$B$. Moreover, we may specify the 
sizes of~$A$ and~$B$ (subject to the trivial necessary conditions).

\begin{lemma}\label{l:nice-split}
Let~$T$ be an oriented tree on~$n$ vertices.
Let~$\calt^-$ be a collection of in-subtrees of~$T$,
and let~$\calt^+$ be a collection of out-subtrees of~$T$,
such that the 
trees in~$\calt^-\cup\calt^+$ are pairwise vertex-disjoint.
If~$a$ and~$b$ are integers with 
\[a\geq \Bigl|\bigcup_{S\in\calt^-}\!\!V(S)\,\Bigr|,
\quad
b\geq \Bigl|\bigcup_{S\in\calt^+}\!\!V(S)\,\Bigr|
\quad\text{and}\quad
a+b=n,
\]
then there exists a partition~$\{A, B\}$ of~$V(T)$ with~$|A|=a$ and~$|B|=b$ such that
$(A, B)$ is a directed pair in~$T$ and
\[
\bigcup_{S\in\calt^-}\!\!V(S) \subseteq A
\quad\text{and}\quad
\bigcup_{S\in\calt^+}\!\!V(S)\subseteq B.
\]
\end{lemma}

\begin{proof}
The key observation is that in every oriented forest there is
a~vertex with no inneighbours (since a forest has more vertices than edges).
Define~$V^-\deq \bigcup_{S\in\calt^-} V(S)$ and~$V^+ \deq\bigcup_{S\in\calt^+} V(S)$, and
let~$k\deq a-|V^-|$, so~$0 \leq k \leq n-|V^-| - |V^+|$.
Greedily choose distinct vertices~$v_1,v_2,\ldots,v_k$ of~$V(T)\setminus (V^-\cup V^+)$
such that~$v_i$ has no inneighbours in~$T - \bigl(V^-\cup V^+\cup \{v_1,\ldots,v_\imm\}\bigr)$
for each~$i\in[k]$. 
The desired partition is then~$A\deq V^-\cup \{v_1,\ldots,v_k\}$ and~$B\deq V(T)\setminus A$.
Indeed, we have~$V^-\subseteq A$,~$V^+\subseteq B$,~$|A|= |V^-| + k = a$ and~$|B| = n- |A| = b$.
It remains to show that~$(A, B)$ is a directed pair in~$T$. So suppose that~$u\rarr v$ is an 
edge of~$T$ and~$v\in A$. It then suffices to show that we must have~$u\in A$ as well.
For this, observe that since~$V^+ \subseteq B$ consists of outstars of~$T$, and~$v \in A$, 
we cannot have~$u \in V^+$. So if~$v \notin V^-$, then~$v = v_i$ for some~$i \in [k]$, and by 
choice of~$v_i$ we then have~$u \in V^- \cup \{v_1, \ldots, v_{i-1}\} \subseteq A$. 
On the other hand, if~$v\in V^-$ then~$v$ is a vertex of some in-subtree of~$T$, so~$u$ must 
be a vertex of the same in-subtree; it follows that~$u\in A$.
\end{proof}

Suppose now that~$T$ is an oriented tree of polylogarithmic maximum degree whose vertex set is 
partitioned into sets~$A$ and~$B$ which form a directed pair~$(A, B)$ in~$T$, and also that~$G$ 
is a tournament whose vertex set admits a partition into sets~$U$ and~$W$ such that~$(U, W)$ is 
an almost-directed pair 
in~$G$. The next lemma shows that if~$U$ and~$W$ are slightly larger than~$A$ and~$B$ respectively, 
then under the additional assumption that every vertex of~$G$ lies in few reverse edges, we may 
embed~$T$ in~$G$ so that vertices of~$A$ are embedded in~$U$ and vertices of~$B$ are embedded 
in~$W$. (Recalling the proof outline of Theorem~\ref{t:good-unavoidable}, we will use this lemma 
to embed the subtree~$T'$ in~$G$.)

\begin{lemma}\label{l:approx-embed-in-cut}
Suppose that~$1/n\ll 1/C$ and that $1/n \ll \mu \ll \alpha$.
Let~$T$ be an oriented tree with~$\Delta(T)\leq (\log n)^C$
and let~$\{A, B\}$ be a partition of~$V(T)$ such that~$(A,B)$ is a 
directed pair in~$T$.
Also let~$G$ be a tournament on~$n$ vertices.
If~$V(G)$ admits a partition~$\{U,W\}$ such that
\begin{enumerate}[label=(\roman*)]
\item \label{i:4.2i} $|U|\geq |A| +\alpha n$,
\item \label{i:4.2ii} $|W|\geq |B| +\alpha n$,
\item \label{i:4.2iii} for each~$u\in U$ we have $\deg^-(u,W)\leq \mu n$, and
\item \label{i:4.2iv} for each~$w\in W$ we have $\deg^+(w,U)\leq \mu n$,
\end{enumerate}
then there exists a copy of~$T$ in~$G$ such that every vertex in~$A$ 
is embedded in~$U$ and every vertex in~$B$ is embedded in~$W$.
\end{lemma}

\begin{proof}
Consider the oriented forest~$F$ with~$F\deq T[A]\cup T[B]$ (in other words,~$V(F) = V(T)$ and 
the edges of~$F$ are the edges of~$T$ with both endvertices in~$A$ or both endvertices in~$B$).
Let~$C_1, \ldots, C_s$ be the components of~$F$, and let~$T'$ be the~minor of~$T$
that we obtain by contracting~$V(C_j)$ to a single vertex~$v_j$, for each~$j\in[s]$.
We may assume the components are labelled 
so that~$v_1,\ldots,v_s$ is an~ancestral ordering of~$V(T')$.
We will greedily embed~$C_1,\ldots,C_s$ in~$G$ in that order, 
defining a~mapping~$\phi\colon V(T)\to U\cup W$.
For each~$j\in[s]$, let~$U_j$ (respectively~$W_j$) be the
set of vertices of~$U$ (respectively~$W$) which have not been covered
by the embedding of~$C_1, \dots, C_{j-1}$.

If~$V(C_1) \subseteq A$, then by~\ref{i:4.2i} we 
have~$|U_1| = |U| \geq |A| + \alpha n \geq |C_1| + \alpha n$, 
so there exists a copy of~$C_1$ in~$G[U_1]$ by Corollary~\ref{cr:S-KMO}.
By a similar argument using~\ref{i:4.2ii} we may embed~$C_1$ in~$G[W_1]$ if~$V(C_1) \subseteq B$.
Now suppose that we have already embedded components~$C_1,\ldots,C_{j-1}$ for some~$1 < j \leq n$, 
so~$\phi(v)$ is defined for every~$v\in \bigcup_{i=1}^{j-1} V(C_i)$. Since we assumed that 
$v_1, \dots, v_s$ was an ancestral ordering of~$V(T')$, 
there exists a unique integer~$i \in [j-1]$ for which some vertex~$u \in V(C_{i})$ is adjacent to
some vertex~$v \in C_j$.
Suppose first that~$u \rarr v\in E(T)$. 
Then~$C_{i}$ has been embedded in~$U$ and~$C_j$ is a component of~$T[B]$,
and we want to embed~$C_j$ in~$W_j\cap N^+\bigl(\phi(u)\bigr)$.
Note that~$\phi(u)$ has at most~$\mu n$ inneighbours in~$W$ by~\ref{i:4.2iii}, 
so by~\ref{i:4.2ii} the number of outneighbours of~$\phi(u)$ in~$W$ which are not in the~image of~$\phi$ (that is,
which are not covered by the embedding so far) is at 
least~$|W_j| - \mu n \geq \bigl(|W| - |B| + |C_j|\bigr) - \mu n \geq |C_j| + \alpha n - \mu n \geq |C_j| +\alpha n/2$.
We may therefore embed~$C_j$ in~$G[W_j]$ by~Corollary~\ref{cr:S-KMO}.
If instead~$u \larr v\in E(T)$ then~$C_j$ is a component of~$T[A]$ and we may embed~$C_j$ in~$G[U_j]$ 
by a similar argument using~\ref{i:4.2i} and~\ref{i:4.2iv}. In either case we have extended~$\phi$ as desired, 
and so proceeding in this manner gives a copy of~$T$ in~$G$.
\end{proof}

We are now ready to state and prove Lemma~\ref{l:embed-nice-in-cut}, the main result of this section,
following the approach sketched in the proof outline of Theorem~\ref{t:good-unavoidable}.

\begin{lemma}\label{l:embed-nice-in-cut}
  Suppose that~$1/n \ll 1/C$ and that $1/n \ll \mu \ll\alpha,\nu$. 
  Let~$G$ be a tournament on~$n$ vertices, and let~$T$ be 
  an~$\alpha$-nice oriented tree on~$n$~vertices with~$\Delta(T)\leq (\log n)^C$.
  If there is a partition~$\{U, W\}$ of~$V(G)$ with~$|U|,|W|\geq \nu n$ such that
  $(U,W)$ is a~$\mu$-almost-directed pair in~$G$,
  then~$G$ contains a~(spanning) copy of~$T$.
\end{lemma}

\begin{proof}
Introduce new constants~$\psi$ and~$\beta$
so that~$1/n \ll \mu \ll \psi \ll \beta \ll \alpha, \nu$.
Since~$(U,W)$ is a~$\mu$-almost directed pair in~$G$, there are at most~$\mu|U||V|$ reverse edges, 
so at~most~$\sqrt{\mu}|U|$ vertices of~$U$ 
are incident to at~least~$\sqrt{\mu}|W|$ reverse edges, 
and at~most~$\sqrt{\mu}|W|$ vertices of~$W$ 
are incident to at~least~$\sqrt{\mu}|U|$ reverse edges. 
Let~$Z$ be the set of all such vertices, so~$z \deq |Z| \leq \sqrt{\mu}\bigl(|U|+|W|\bigr) = \sqrt{\mu}n$.
Now let~$W_0 \deq W \sm Z$, and let~$X$ be the set of all vertices~$w \in W_0$
with~$\deg^0(w, W_0) < \psi n$. Then by~Lemma~\ref{l:semidegree} we have~$|X| < 4\psi n$.
Choose a~subset~$Y \subseteq W_0$ of size~$\psi n$ uniformly at~random. 
Note that for each~$w\in W_0 \setminus X$
the values of~$\deg^-(w,Y)$ and of~$\deg^+(w,Y)$
then have a hypergeometric distribution
with expectation at~least~$\psi n |Y|/|W_0|\geq \psi^2 n$,
so~$\bbp(\,\deg^0(w,Y) < \psi^2 n/2\,)$ decreases exponentially with~$n$
by~Theorem~\ref{t:exp}.
Taking a union bound over the at most~$n$ vertices~$w \in W_0 \setminus X$ we find that 
with positive probability every~$w\in W_0 \setminus X$
has~$\deg^0(w, Y) \geq \psi^2 n/2 \geq 2z$. 
Fix a choice of~$Y$ for which this event occurs
and define~$U' \deq U \setminus Z$ and~$W' \deq W_0\setminus(Y\cup X)$. 
Also let~$n' \deq |U' \cup W'|$, so~$n' \geq n - |X| - |Y| - |Z| \geq (1 - 6\psi) n$. 
Observe that we then have the following properties.
\begin{enumerate}[label=(\alph*)]
\item \label{i:4.3i} Every vertex~$u \in U \sm Z$ has~$\deg^-(u, W') \leq \sqrt{\mu}|W| \leq \psi n'$.
\item \label{i:4.3ii} Every vertex~$w \in W \sm Z$ has~$\deg^+(w, U') \leq \sqrt{\mu}|U| \leq \psi n'$.
\item \label{i:4.3iii} Every vertex~$w \in W'$ has~$\deg^0(w, Y) \geq 2z$.
\item \label{i:4.3iv} $|U'| \geq |U| - |Z| \geq |U| - \sqrt{\mu}n$ and $|W'| \geq |W| - |X| - |Y| - |Z| \geq |W| - 6\psi n$. 
\item \label{i:4.3v} $\Delta(T) \leq (\log n)^C \leq (\log n')^{2C}$.
\end{enumerate}

Define~$t\deq \lceil \beta n \rceil$. Let~$\cals^-$ be the set of pendant instars of~$T$ which 
contain an out-leaf of~$T$, and let~$\cals^+$ be the set of pendant outstars of~$T$ which contain both 
an in-leaf of~$T$ and an out-leaf of~$T$. Observe that~$\cals^- \cup \cals^+$ is then a set of 
vertex-disjoint subtrees of~$T$. Moreover, since~$T$ is~$\alpha$-nice, 
we have~$|\cals^-|, |\cals^+| \geq \alpha n$. We define~$S_1^-, \ldots, S_t^-$ to be the 
smallest~$t$ members of~$\cals^-$ and~$S_1, \ldots, S_{t+z}^+$ to be the smallest~$t+z$ members 
of~$\cals^+$. Since~$t+z \leq 2\beta n$ we must then 
have~$\bigl|\bigcup_{i \in [t]} V(S_i^-)\bigr|, \bigl|\bigcup_{i \in [t+z]} V(S_i^+)\bigr| \leq 2\beta n/ \alpha$. 
For each~$i \in [t]$ let~$\ell_i^+$ be an out-leaf of~$T$ in~$S_i^-$ 
and let~$c_i^-$ be the centre of the star~$S_i^-$, and for 
each~$i \in [z]$ let~$\ell_{t+i}^+$ be an out-leaf of~$T$ in~$S_{t+i}^+$. 
Similarly, for each~$i \in [t+z]$ let~$\ell_i^-$ be an in-leaf of~$T$ 
in~$S_i^+$ and let~$c_i^+$ be the centre of the star~$S_i^+$. We can be sure that these leaves exist 
by definition of~$\cals^+$ and~$\cals^-$.

We now define~$T'$ to be the subtree of~$T$ obtained by deleting the 
leaves~$\ell_i^+$ and~$\ell_i^-$ from~$T$ for each~$i \in [t+z]$.
So~$L^-_i \deq S_i^--\ell_i^+$ (respectively~$L_i^+ \deq S_i^+-\ell_i^-$) is an in-subtree 
(respectively out-subtree) of~$T'$ for each~$i \in [t]$, and 
$L_{t+j}^+ \deq S_{t+j}^+ -\{\ell_{t+j}^-,\ell_{t+j}^+\}$ is an out-subtree of~$T'$ for 
each~$j \in [z]$.
Also define~$a \deq |U|-t$ and~$b \deq |W|-t-2z$. Then 
we have~$a \geq \nu n - t \geq 2\beta n/\alpha \geq \bigl|\bigcup_{i \in [t]} V(L_i^-)\bigr|$ 
and~$b \geq \nu n - t - 2z \geq 2\beta n/\alpha \geq \bigl|\bigcup_{i \in [t+z]} V(L_i^+)\bigr|$, 
and also~$a+b = |U| + |W| - 2t-2z = |T'|$, so we may apply 
Lemma~\ref{l:nice-split} to obtain a partition~$\{A, B\}$ of~$V(T')$ with~$|A| = a$
and~$|B| = b$ such that~$(A, B)$ is a directed pair in~$T'$ and so that~$V(L_i^-) \subseteq A$ 
for each~$i \in [t]$ and~$V(L_i^+) \subseteq B$ for each~$i \in [t+z]$. Next, since by~\ref{i:4.3iv} 
we have~$|U'| \geq a + \beta n'/2$ and~$|W'| \geq b + \beta n'/2$, 
by~\ref{i:4.3i},~\ref{i:4.3ii} and~\ref{i:4.3v} we may apply Lemma~\ref{l:approx-embed-in-cut} 
(with~$n', \psi, 2C$ and~$\beta/2$ in place of~$n, \mu, C$ and~$\alpha$ respectively)
to obtain an embedding~$\phi$ of~$T'$ in~$G$ so that~$\phi(A) \subseteq U'$ and 
$\phi(B) \subseteq W'$. 

We next embed the vertices~$\ell_{t+j}^+$ and~$\ell_{t+j}^-$ for~$j \in [z]$ so that 
all vertices of~$Z$ are covered. Note  that our embedding of~$T'$ in~$G$ ensured that
for each~$j\in[z]$ the centre~$c_{t+j}^+$ of~$S_{t+j}^+$ was embedded to
a~vertex~$w_{t+j} \deq \phi(c_{t+j}^+)$ in~$W'$, so in particular we 
have~$\deg^0(w_{t+j},Y) \geq 2z$ by~\ref{i:4.3iii}.
This means that we can greedily choose distinct vertices~$y_1^-,y_1^+,\ldots,y_z^-,y_z^+ \in Y$
so that for each~$j\in[z]$ the vertex~$y_j^-$ is an inneighbour of~$w_{t+j}$
and~$y_j^+$ is an outneighbour of~$w_{t+j}$.
Write~$Z \deq \{q_1,\ldots,q_z\}$, and for each~$j\in[z]$ consider the orientation of the edge 
of~$G$ between~$q_j$ and~$w_{t+j}$.
If~$q_j\rarr w_{t+j}\in E(G)$, then we set~$\phi(\ell_{t+j}^-) \deq q_j$ and~$\phi(\ell_{t+j}^+)\deq y_j^+$. 
Similarly, if~$q_j\larr w_{t+j}\in E(G)$, then we set~$\phi(\ell_{t+j}^-) \deq y_j^-$ 
and~$\phi(\ell_{t+j}^+)\deq q_j$.

Observe that we have now embedded all of the vertices of~$T$ except for the 
leaves~$\ell_1^+,\ldots,\ell_t^+$ and~$\ell_1^-,\ldots,\ell_t^-$.
Let~$P^- \deq \{\,\phi(c_i^-) : i\in[t]\,\}$ and~$P^+ \deq \{\,\phi(c_i^+) : i\in[t]\,\}$, 
so~$P^- \subseteq U'$ and~$P^+ \subseteq W'$. 
Also, let~$Q^-$ be the set of uncovered vertices of~$U$ and let~$Q^+$ be the set of uncovered vertices
of~$W$. Then~$|Q^-| = |U| - a = t$, and~$|Q^+| = |W| - b - 2z =t$, so 
we have~$|P^-| = |P^+| = |Q^-| = |Q^+| = t$. Observe that since we already covered all vertices of~$Z$, 
we also have~$Q^- \subseteq U \sm Z$ and~$Q^+ \subseteq W \sm Z$. Together with the fact 
that~$t =\lceil \beta n \rceil$, by~\ref{i:4.3i} and~\ref{i:4.3ii} it follows 
that~$G[P^- \rarr Q^+]$ and~$G[Q^- \rarr P^+]$ are both~$(1, \frac{1}{2})$-super-regular, 
so the balanced bipartite underlying graph of each contains a perfect matching by Lemma~\ref{l:matching}. 
For each~$j \in [t]$ let~$\phi(\ell_j^+) \in Q^+$ (respectively~$\phi(\ell_j^-) \in Q^-$) be the vertex 
matched to~$\phi(c_j^-) \in P^-$ (respectively~$\phi(c_j^+) \in P^+$); 
this completes the embedding~$\phi$ of~$T$ in~$G$.
\end{proof}

\section{Cycles of cluster tournaments}

\label{s:embed-in-CCT}

Our goal in this section is to prove Lemma~\ref{l:in-CCT}, which states that every
sufficiently large tournament
containing an almost-spanning regular~\CCT\ contains a spanning 
copy of every nice oriented tree~$T$ with polylogarithmic maximum degree. 
Recall from the proof sketch of Theorem~\ref{t:good-unavoidable} that for this we 
split~$T$ into two subtrees~$T_1$ and~$T_2$. We then embed~$T_1$ so that all `atypical' 
vertices are covered and so that roughly the same number of vertices from each cluster 
are covered. Since~$T_1$ covered all atypical vertices, the vertices which remain uncovered then
form a super-regular \CCT, and we use this fact to embed~$T_2$ to cover all vertices which remain uncovered 
and so complete the embedding of~$T$ in~$G$. In Section~\ref{s:1st} we focus on the 
embedding of~$T_1$, showing that can find an embedding with the desired 
properties (Lemma~\ref{l:1st-half}). Likewise, in Section~\ref{s:2nd} we consider the 
embedding of~$T_2$, and prove that we can indeed embed~$T_2$ so as to 
cover all remaining vertices, as desired (Lemma~\ref{l:2nd-half}). 
Finally, in Section~\ref{s:joining} we combine these
results to prove Lemma~\ref{l:in-CCT} by first splitting~$T$ into subtrees~$T_1$ and~$T_2$ and then
successively embedding these subtrees using Lemmas~\ref{l:1st-half} and~\ref{l:2nd-half}. 

\subsection{Embedding the first subtree} \label{s:1st}

The subtree~$T_1$ will have polylogarithmic maximum degree and will
contain many vertices which are adjacent to
at least one in-leaf and at least one out-leaf of~$T$, and
we wish to embed~$T_1$ into a tournament~$G$ which contains an almost-spanning \CCT\ so
that approximately the same number of vertices of~$T_1$ are embedded in each cluster.
The following lemma states that we can indeed do this.

\begin{lemma}\label{l:1st-half}
  Suppose that~$1/n \ll 1/C$ and that
  $1/n\ll 1/k \ll \eps \ll d\ll\psi\ll\beta\ll \alpha$.
  Let~$T$ be an~oriented tree on~$n$~vertices with root~$r$, with 
  maximum degree~$\Delta(T)\leq (\log n)^C$, and 
  which contains at least~$\beta n$~distinct vertices that are each adjacent to at 
  least one in-leaf and at least one out-leaf of~$T$.
  Let~$G$ be a tournament which contains a~$(d, \eps)$-regular \CCT\ whose 
  clusters~$V_1, \dots, V_k$ have size~$(1+\alpha)\frac{n}{k} \leq |V_i| \leq \frac{3n}{k}$ for each \ink,
  and assume additionally that~$B \deq V(G) \sm \bigcup_\ink V_i$ has size~$|B| \leq \psi n$.
  Then there exists an~embedding~$\phi$ of~$T$ in~$G$
  covering~$B$, such that~$r$ is embedded in~$V_1$ 
  and such that for each~\ink\ we have
  \[
  \bigl|\phi\bigl(V(T)\bigr) \cap V_i\bigr| 
  = \bigl(n-|B|\bigr)\left(\frac{1}{k}\pm  {\frac{2}{\log\log n}}\right).
  \]
\end{lemma}

Loosely speaking the proof proceeds as follows.
We begin by~selecting from each cluster~$V_i$ 
a~large subset~$V_i'$ of~vertices which each
have large semidegree in~$V_i\setminus V_i'$.
Then~$V_1', \dots, V_k'$ are the clusters of a regular \CCT\ in~$G'\deq G\bigl[\bigcup_\ink V_i'\bigr]$.
We~remove a small number of leaves from~$T$ to obtain a subtree~$T'$, 
and embed~$T'$ in~$G'$ by using the Vertex Allocation Algorithm (Algorithm~\ref{a:alloc}) and 
Lemma~\ref{l:approx-embedding}. Lemma~\ref{l:3.3-stren} then ensures that 
approximately the same number of vertices are embedded in each cluster.
Finally, we extend the embedding of~$T'$ in~$G$ to an embedding of~$T$ in~$G$ by embedding 
the removed leaves so as to cover all vertices of~$B$.

\begin{proof}
Define~$m \deq \frac{n}{k}$, so~$(1+\alpha)m \leq |V_i| \leq 3m$ for each~\ink, 
and let~$\delta\deq {\frac{1}{\log\log n}}$.
Let~$B_i$ be the set of all vertices~$x\in V_i$
such that~$\deg^0(x,V_i) < \alpha m/20$.
By~Lemma~\ref{l:semidegree} we have~$|B_i| < \alpha m/4$.
For each~\ink, pick a~subset~$Y_i\subseteq V_i$
of size~$|Y_i| = \alpha m/4$ uniformly at~random
with choices made independently for each~$i$.
Note that for each~\ink\ and each~$x\in V_i\setminus B_i$,
the random variables~$\deg^-(x,Y_i)$ and~$\deg^+(x, Y_i)$ then
have hypergeometric distributions
with expected value at~least~$(\alpha m/20) |Y_i|/|V_i| > 5\beta m$,
and thus~$\bbp(\,\deg^0(x,Y_i) < 4\beta m\,)$ decreases exponentially with~$n$
by~Theorem~\ref{t:exp}.
Taking a union bound, we find that there is a positive probability that for every~\ink\
and every~$x \in V_i\setminus B_i$ we have~$\deg^0(x, Y_i) \geq 4\beta m$. Fix a 
choice of sets~$Y_1,\ldots,Y_k$ such that this event occurs, and for each~\ink\
let~$V_i' \deq V_i \setminus (Y_i \cup B_i)$, so 
\[3m \geq |V_i| \geq |V_i'| \geq |V_i| - |B_i| -|Y_i| > 
(1 + \alpha)m - \frac{\alpha m}{4} - \frac{\alpha m}{4} = \left(1+\frac{\alpha}{2}\right)m.\]
Furthermore, every vertex~$x \in V'_i$ has~$\deg^0(x, Y_i) \geq 4\beta m$.
Now define~$G' \deq G[V_1' \cup \dots \cup V_k']$, and observe that since~$V_1, \dots, V_k$ were the 
clusters of a~$(d,\eps)$-regular \CCT\ in~$G$, by Lemma~\ref{l:slice-pair} 
the sets~$V'_1, \dots, V'_k$ are the clusters of a spanning~$(d, 3\eps)$-regular \CCT\ in~$G'$.
In particular we may choose a vertex~$v \in V'_1$ with at least~$(d-3\eps)|V'_k|$ inneighbours in~$V'_k$ 
and at least~$(d-3\eps)|V'_2|$ outneighbours in~$V'_2$. 
The tournament~$G'$, the clusters~$V'_1, \dots, V'_k$ and the vertex~$v$ then meet 
the conditions of Lemma~\ref{l:approx-embedding} 
with~$\alpha/2$ and~$3\eps$ in place of~$\alpha$ and~$\eps$ respectively 
(and with~$n$ playing the same role there as here).

Let~$t\deq \lceil \beta n\rceil - 1$, and 
choose a~set~$W\deq \{w_1,\ldots,w_t\}$ of~$t$ 
distinct vertices in~$T$ so that each~$w_i$ is adjacent to at least
one~in-leaf and at least one~out-leaf of~$T$ and so that~$r$ is not a leaf of~$T$ which is adjacent to 
a vertex of~$W$ (such a set exists by the assumptions of the lemma).
For each~$j\in[t]$, let~$w_j^-$ and~$w_j^+$ be respectively
an~in-leaf and an~out-leaf adjacent to~$w_j$.
Let~$T'$ be the oriented tree we obtain by deleting from~$T$ the
vertices~$w_j^-$ and~$w_j^+$ for each~$j\in[t]$, so~$|T'| = n-2t$ 
and~$\Delta(T') \leq \Delta(T) \leq (\log n)^C \leq (\log (n-2t))^{2C}$. 
Also take~$r$ to be the root of~$T'$, and
apply the Vertex Allocation Algorithm (Algorithm~\ref{a:alloc}) to allocate the vertices of~$T'$ to 
the clusters~$V'_1, \dots, V'_k$. 
By~Lemma~\ref{l:3.3-stren}\ref{l:3.3-stren-i} the obtained allocation will be semi-canonical. 
Moreover, by two applications of Lemma~\ref{l:3.3-stren}\ref{l:3.3-stren-iv} 
(with~$\beta/2$ and~$2C$ in place of~$\alpha$ and~$C$ respectively) we have with 
probability~$1-\littleo(1)$ that for each~\ink\ the number of vertices of~$T'$ allocated to the 
cluster~$V'_i$ is 
\begin{equation} \label{e:wbound}
(n-2t)\left(\frac{1}{k} \pm \frac{1}{\log \log (n-2t)}\right) = \frac{n-2t}{k} \pm \frac{3 \delta n}{2},
\end{equation}
and the number of vertices of~$W$ allocated to the cluster~$V'_i$ is 
\begin{equation} \label{e:tbound}
 t \left(\frac{1}{k} \pm \frac{1}{\log \log (n-2t)}\right) = \frac{t}{k} \pm \frac{3\delta t}{2}.
 \end{equation}
Fix an outcome of the Vertex Allocation Algorithm for which each of these events occurs, and 
apply~Lemma~\ref{l:approx-embedding}
to~obtain an~embedding~$\phi$ of~$T'$ in~$G'$
so that~$r$ is embedded to~$v$ and each vertex of~$T'$ is embedded in the cluster~$V'_i$
to which it is allocated. In particular~$r$ is embedded in~$V_1$, as required.

We now extend~$\phi$ to an~embedding of~$T$ in~$G$ which covers~$B$.
Let~$b\deq |B|\leq \psi n$, and let~$q_1,\ldots,q_b$ be the vertices of~$B$. 
Also let~$p\in[k]$ be such that~$b \equiv p\bmod k$, and for each~$i\in[k]$
choose~$W_i\subseteq W$ such that~$\phi(W_i)\subseteq\phi(W)\cap V_i'$ and so that
$|W_i| = \lceil b/k\rceil$ if~$i \in [p]$ and~$|W_i|= \lfloor b/k\rfloor$ if~$i\in[k]\setminus[p]$. 
(Since~$b/k \leq \psi n/k$ and~$\psi \ll \beta$,~\eqref{e:tbound} ensures that we can indeed choose 
such sets.) The sets~$W_1,\ldots,W_k$ are then vertex-disjoint and~$|\bigcup_\ink W_i| = b$, so
by relabelling if necessary we may assume that~$\bigcup_\ink W_i =\{w_1,\ldots,w_b\}$.
For each~$j\in[t]$ set~$p_j \deq \phi(w_j)$ and write~$i_j$ to 
denote the index such that~$p_j \in V_{i_j}$.
Greedily choose~$2t$~distinct vertices~$c_1^-,c_1^+,\ldots,c_t^-,c_t^+$ so that for each~$j \in [t]$ 
we have that~$c_j^-,c_j^+\in Y_{i_j}$, that~$c_j^-$ is an inneighbour of~$p_j$
and that~$c_j^+$ is an outneighbour of~$p_j$. It is possible to make such choices since for each~\ink\
there are at most~$2t/k$ vertices~$w_j$ with~$i_j = i$ by~\eqref{e:tbound},
and because for each~$j \in [t]$ we have~$p_j \in V'_{i_j}$ (since~$w_j$ is a vertex of~$T'$), 
so the~semidegree of~$p_j$ in~$Y_{i_j}$
is at~least~$4 \beta m \geq 2\cdot (2t/k)$ by our choice of the sets~$Y_i$.

Recall that each vertex in~$W$ is adjacent to precisely one removed in-leaf~$w_j^-$ of~$T$ and one removed 
out-leaf~$w_j^+$ of~$T$, and that these leaves have not yet been embedded.
For each~$s \in [b]$ we embed one of these leaves to the vertex~$q_s$ and the other to either
$c_s^-$ or~$c_s^+$ according to the direction of the edge between~$q_s$ and~$p_s$. 
For each~$b+1 \leq s \leq t$ we then embed the in-leaf of~$w_s$ to~$c_s^-$ and the 
out-leaf of~$w_s$ to~$c_s^+$.
More precisely,
for all integers~$s$ with~$1\leq s\leq b$
we set~$\phi(w_s^-)\deq q_s$ and~$\phi(w_s^+)\deq c_s^+$ if~$q_s\rarr p_s\in E(G)$,
and set~$\phi(w_s^+)\deq q_s$ and~$\phi(w_s^-)\deq c_s^-$ if~$q_s\larr p_s \in E(G)$.
Then, for all integers~$s$ with~$b < s \leq  t$
we set~$\phi(w_s^-) \deq  c_s^-$ and~$\phi(w_s^+)\deq c_s^+$.
Following this extension~$\phi$ is an embedding of~$T$ in~$G$ which covers every vertex in~$B$.
Moreover, for each~$i \in [k]$ the number of vertices embedded in the cluster~$V_i$ is
\[
\bigl|\phi\bigl(V(T)\bigr) \cap V_i\bigr| = \left(\frac{n-2t}{k} \pm \frac{3\delta n}{2}\right)
+ 2\left(\frac{t}{k} \pm \frac{3\delta t}{2}\right) 
- \left(\frac{b}{k} \pm 1 \right)
=
\bigl(n-|B|\bigr)\left(\frac{1}{k}\pm 2\delta\right)
\]
where the first term counts the number of vertices of~$T'$ embedded in~$V_i$ (see \eqref{e:wbound}), 
and the second and third terms count the number of removed leaves embedded in~$V_i$.
Indeed, by~\eqref{e:tbound} there are~$t/k \pm 3\delta t/2$ vertices of~$W$ embedded in~$V_i$, 
each of which is adjacent to two removed leaves, and these removed leaves are each embedded 
in~$V_i$ except for the~$\lfloor b/k\rfloor$ or~$\lceil b/k\rceil$ leaves embedded in~$B$.
\end{proof}

\subsection{Embedding the second subtree}\label{s:2nd}

Recall from the outline at the beginning of this section that, following the embedding of the 
first subtree~$T_1$, the vertices which remain uncovered form a super-regular \CCT. We wish to 
embed the second subtree~$T_2$ so that all of these vertices are covered. The following lemma
demonstrates that this is possible.

\begin{lemma}\label{l:2nd-half}
Suppose that~$1/n \ll 1/C$ and that~$1/n \ll 1/k \ll \eps \ll d \ll \beta$. 
Let~$T$ be an~oriented tree on~$n$ vertices with root~$r$, with 
maximum degree~$\Delta(T)\leq (\log n)^C$, and which contains at~least~$\beta n$~distinct 
vertices that are each adjacent to at least one in-leaf and at least one out-leaf of~$T$.
Let~$G$ be a~$(d,\eps)$-super-regular \CCT\ on~$n$ vertices whose
clusters~$V_1,\ldots, V_k$ each have size~$\frac{n}{k} \pm \frac{2n}{\lln}$,
and let~\vstar\ be a vertex of~$V_1$.
Then~$G$ contains a (spanning) copy of~$T$ 
in which~$r$ is embedded to~$v$.
\end{lemma}

Loosely speaking, the proof of Lemma~\ref{l:2nd-half} begins by removing a small
number of in-leaves and out-leaves of~$T$ to obtain a subtree~$T'$. 
We then select small disjoint subsets~$X_i$ and~$Y_i$ of~$V_i$ for each~\ink\
with the property that each vertex in~$V_i$ has
many inneighbours in each of~$X_\imm$ and~$Y_\imm$ and
many outneighbours in each of~$X_\ipp$ and~$Y_\ipp$,
and so that most vertices in~$V_i$
have large semidegree in~$X_i$.
Removing these sets from~$G$ yields a subgraph~$G'$ of~$G$ which
is a regular \CCT, and we embed~$T'$ in~$G'$ 
using Lemmas~\ref{l:3.3-stren} and~\ref{l:approx-embedding}. It remains
to embed the removed leaves of~$T$ so as to cover all vertices of~$G$ which
remain uncovered. We first use the fact that the image of each vertex of~$T'$ embedded 
in~$V_i$ has large semidegree in~$X_i$ to embed a small number of removed 
leaves to equalise the numbers of uncovered vertices in each cluster and 
the numbers of removed leaves needing to be embedded in that cluster, 
before completing the embedding by using the super-regularity of~$G$ to 
find perfect matchings in appropriate auxiliary bipartite graphs. 

\begin{proof}
Introduce new constants~$\eta$ and~$\gamma$ such that
$\eps \ll\eta \ll \gamma \ll d$. Also define~$\delta \deq  {\frac{2}{\lln}}$ 
and~$m \deq \frac{n}{k}$, so each cluster has size~$m \pm \delta n$, assume without 
loss of generality that~$\beta \leq \frac{1}{4}$, and let~$t \deq \lceil \beta n \rceil -1$.
Choose a set~$W$ of~$t$ distinct vertices of~$T$ 
so that each~$w \in W$ is adjacent to at least one in-leaf of~$T$ and at least one out-leaf of~$T$ 
and so that~$r$ is neither in~$W$ nor a leaf of~$T$ which is adjacent to a vertex of~$W$
(our assumption on~$T$ ensures that we can choose such a set~$W$).
Let~$T'$ be the oriented tree formed by deleting from~$T$ precisely one in-leaf and 
one out-leaf adjacent to each vertex of~$W$, and take~$r$ to be the root of~$T'$.
Observe that~$T'$ then has precisely~$n-2t$ vertices and 
maximum degree~$\Delta(T') \leq \Delta(T) \leq (\log n^C) \leq \bigl(\log (n-2t)\bigr)^{2C}$; 
in other words,~$T'$ meets the conditions of Lemma~\ref{l:approx-embedding} 
with~$n-2t$ and~$2C$ in place of~$n$ and~$C$ respectively.
We will embed~$T'$ in an appropriate subgraph of~$G$, which we find by using the following claim.

\begin{claim}\label{c:reserve-XY}
For each~$i \in [k]$ there exist sets~$F_i, X_i, Y_i \subseteq V_i$ with~$X_i, Y_i \subseteq F_i$ 
such that, writing~$V_i' \deq V_i\setminus F_i$, we have
\begin{enumerate}[label=(\roman*)]
\item $|F_i| \leq 3 \gamma m$,
\item $X_i$ and~$Y_i$ are disjoint, and~$v \in V'_1$,
\item for each~$x\in V_i'\setminus\{\vstar\}$ we have~$\deg^0(x,X_i) \geq \eta m$, and \label{i:deg-X}
\item for each~$x\in V_i$ 
we have~$\deg^-(x,X_\imm), \deg^-(x,Y_\imm), \deg^+(x,X_\ipp), \deg^+(x,Y_\ipp) \geq \eta m$.\label{i:deg-Y}
\end{enumerate}
\end{claim}

\begin{claimproof}
For each \ink\ let~$D_i\subseteq V_i$ consist of all vertices~$x \in V_i$ with~$\deg^0(x, V_i) < \gamma m/5$,
 so~$|D_i| \leq \gamma m$ by~Lemma~\ref{l:semidegree}. 
Then every vertex~$x \in V_i \sm D_i$ has~$\deg^0(x, V_i) \geq \gamma m/5$. 
Also, since~$V_1, \dots, V_k$ are the clusters of a~$(d, \eps)$-super-regular \CCT, 
every vertex~$x \in V_i$ has at least~$(d - \eps)|V_\imm| \geq dm/2$ inneighbours in~$V_\imm$ 
and at least~$(d-\eps)|V_\ipp| \geq dm/2$ outneighbours in~$V_\ipp$.
For each~\ink\ choose disjoint subsets~$X_i, Y_i \subseteq V_i$ 
with~$|X_i| = |Y_i| =\lfloor \gamma m \rfloor$ uniformly at random and independently of all other choices.
Then for each~\ink\ and each~$x \in V_i$ the random 
variables~$\deg^-(x, X_\imm), \deg^-(x, Y_\imm), \deg^+(x, X_\ipp)$ and~$\deg^+(x, Y_\ipp)$ each 
have hypergeometric distribution with expectation 
at least~$(dm/2)\lfloor \gamma m \rfloor/(m+\delta n) \geq d\gamma m/3 \geq 2\eta m$; if additionally~$x \in V_i\setminus D_i$,
then the random variables~$\deg^+(x, X_i)$ and~$\deg^-(x, X_i)$ each have hypergeometric distribution with 
expectation at least~$(\gamma m/5)\lfloor \gamma m \rfloor/(m + \delta n) \geq \gamma^2 m/6 \geq 2\eta m$. The probability that any given 
one of these random variables is less than~$\eta m+1$ therefore declines exponentially with~$n$ by 
Theorem~\ref{t:exp}, so by taking a union bound over all of these at most~$6n$ events we find that with 
positive probability none of these random variables is less than~$\eta m+1$. Fix a choice of the sets~$X_i$ 
and~$Y_i$ with this property, then removing~$\vstar$ from the sets~$X_1$ and~$Y_1$ if necessary and taking~$F_1 \deq (X_1 \cup Y_1 \cup D_1) \sm \{\vstar\}$ 
and~$F_i \deq X_i \cup Y_i \cup D_i$ for each~$2 \leq i \leq k$ gives the desired sets. 
\end{claimproof}

Fix sets~$F_i, X_i, Y_i$ and~$V_i'$ as in Claim~\ref{c:reserve-XY}, and observe that
\[|V'_i| = |V_i| - |F_i| \geq  m - \delta n - 3\gamma m \geq \left(1-\frac{\beta}{4}\right) m 
\geq \left(1+\beta\right)\frac{n-2t}{k}.\]
Let~$G'\deq G[V_1'\cup\cdots\cup V_k']$, and note that
by Lemma~\ref{l:slice-pair}~$G'$ is then a~$(d,2\eps)$-regular \CCT\ with clusters~$V_1',\ldots,V_k'$. 
Observe also that~$\vstar$ has at least~$(d-\eps)|V_2| - |F_2| \geq \gamma m$ outneighbours in~$V'_2$ and 
at least~$(d-\eps)|V_k| - |F_k| \geq \gamma m$ inneighbours in~$V'_k$.
In other words,~$G'$ meets the conditions of Lemma~\ref{l:approx-embedding} with~$n-2t$,~$\beta$ 
and~$2\eps$ in place of~$n$,~$\alpha$ and~$\eps$ respectively
(so, in particular,~$m$ there corresponds to~$m - 2t/k$ here).

Apply the Vertex Allocation Algorithm (Algorithm~\ref{a:alloc}) to allocate the vertices of~$T'$ to the 
clusters~$V'_1, \dots, V'_k$ of~$G'$. For each~$i \in [k]$ let~$T'_i$ consist of all vertices of~$T'$ 
allocated to the cluster~$V'_i$, and likewise let~$W_i \subseteq W$ consist of all vertices of~$W$ 
allocated to the cluster~$V'_i$. By Lemma~\ref{l:3.3-stren}\ref{l:3.3-stren-i} the allocation we obtain from 
the Vertex Allocation Algorithm will be semi-canonical. Furthermore, by two applications of 
Lemma~\ref{l:3.3-stren}\ref{l:3.3-stren-iv} (with~$n-2t$,~$2C$ and~$\beta/2$ in place of~$n$,~$C$ 
and~$\alpha$ respectively) we find with probability~$1-\littleo(1)$ that for every~$i \in [k]$ we have
\begin{equation} \label{e:wuniform}
|W_i| \deq |W| \left(\frac{1}{k} \pm \frac{1}{\log \log (n-2t)}\right) = \frac{t}{k} \pm \delta n
\end{equation}
and 
\begin{equation} \label{e:tuniform}
|T'_i| \deq (n-2t)\left(\frac{1}{k} \pm \frac{1}{\log \log (n-2t)}\right) = m - \frac{2t}{k} \pm \delta n.
\end{equation}
Fix an allocation with these properties, and observe that this allocation then meets the conditions of 
Lemma~\ref{l:approx-embedding} with~$n-2t$ and~$\beta$ in place of~$n$ and~$\alpha$ respectively. 
So we may apply Lemma~\ref{l:approx-embedding} to obtain an embedding~$\phi$ of~$T'$ in~$G'$
such that each vertex of~$T'$ is embedded in the cluster to which it was allocated and 
so that~$\phi(r) = \vstar$.

For each~\ink, let~$U_i \subseteq V_i$ be the set of vertices of~$V_i$ not covered by~$\phi$
and let~$U \deq \bigcup_\ink U_i$. Then, since every vertex was embedded in 
the cluster to which it was allocated, by~\eqref{e:tuniform} we have for each~\ink\ that 
\begin{equation} \label{e:uuniform}
|U_i| = |V_i| -|T'_i| = (m \pm \delta n) - \left(m - \frac{2t}{k} \pm \delta n\right) 
= \frac{2t}{k}\pm 2\delta n,
\end{equation}
and since~$|G| = n$ and~$|T'| = n-2t$ we have~$|U| = 2t$. 
Also, for each~\ink, let~$P_i \deq \phi(W_i)$ and write~$P \deq \bigcup_\ink P_i$. 
In other words,~$P_i$ (respectively~$P$) is the is the set of vertices of~$G$ to which 
vertices of~$W_i$ (respectively~$W$) were embedded. 
So~$P_i \subseteq V'_i$ and~$|P_i| = |W_i|$, and similarly~$|P| = |W| = t$.

Our goal for the remaining part of the proof is to choose, 
for each~$x \in P$, an inneighbour~$x^-$ of~$x$ in~$U$ and an outneighbour~$x^+$ of~$x$ in~$U$ such that
the chosen inneighbours and outneighbours are all distinct. Indeed, for each vertex~$w \in W$ there
is a unique vertex~$x \in P$ with~$\phi(w) = x$. Let~$w^+$ and~$w^-$ denote the out-leaf and in-leaf 
adjacent to~$w$ which we removed when forming~$T'$; we could then embed~$w^+$ to~$x^+$ and~$w^-$ to~$x^-$, 
and doing so for each~$w \in W$ would extend~$\phi$ to an embedding of~$T$ in~$G$, completing the proof.
If for every~$\ink$ both 
$G[U_\imm \rarr P_i]$ and~$G[P_i\rarr U_\ipp]$ are super-regular 
and~$|U_i| = |P_\imm|+|P_\ipp|$, then (after appropriately partitioning 
the sets~$U_i$) we could apply Lemma~\ref{l:matching} to find, for each~$i \in [k]$ and each~$x \in P_i$, 
vertices~$x^- \in U_\imm$ and~$x^+ \in U_\ipp$ satisfying the above properties. 
However, neither of these assumptions is necessarily valid. Over the following steps of the proof we embed 
the removed leaves adjacent to a small number of vertices of~$W$ so that these assumptions do indeed hold 
for the remaining vertices; we then complete the embedding of~$T$ in~$G$ in the manner described above.

\medskip
\emph{Step 1: Balancing the sets~$W_i$ and ensuring super-regularity.} The first step of this process is to 
embed the removed leaves adjacent to a small number of vertices of~$W$ so that equally many vertices in each 
set~$W_i$ have not had their adjacent removed leaves embedded. We also cover all vertices in each set~$U_i$ 
which have too few inneighbours in~$P_\imm$ or too few outneighbours in~$P_\ipp$; this will ensure that the 
auxiliary bipartite graphs which we consider at the end of the proof are super-regular.

For each~$\ink$ define~$s_i \deq \lfloor 4\eps m\rfloor + |W_i| - \min_{i \in [k]} |W_i|$, so 
by~\eqref{e:wuniform} we have~$\lfloor 4\eps m\rfloor \leq s_i \leq 4\eps m + 2\delta n$.
Also, for each~\ink, 
let~$B_i^-$ be the set of~vertices in~$U_i$ with fewer than~$\eta m$ inneighbours in~$P_\imm$, and
let~$B_i^+$ be the set of~vertices in~$U_i$ with fewer than~$\eta m$ outneighbours in~$P_\ipp$.
Since~$G[V_\imm \rarr V_i]$ is~$(\dplus,\eps)$-regular, and~$|P_\imm| = |W_\imm| \geq t/k - \delta n > \eps |V_\imm|$ by~\eqref{e:wuniform}, 
we must have~$|B_i^-| \leq \eps |V_i| \leq 2 \eps m$; 
likewise, since~$G[V_i\rarr V_\ipp]$ is~$(\dplus,\eps)$-regular 
and~$|P_\ipp| > \eps |V_\ipp|$, we must 
have~$|B_i^+| \leq \eps|V_i| \leq 2 \eps m$. So we may choose for 
each~\ink\ a subset~$B_i \subseteq U_i$ of size~$|B_i| = s_i$ with~$B_i^-\cup B_i^+ \subseteq B_i$.

Next, for each~\ink\ we proceed as follows. 
Let~$\{b_1,\ldots,b_{s_i}\}$ be the vertices in~$B_i$, arbitrarily choose distinct
vertices~$w_1, \dots, w_{s_i} \in W_i$, and for each~$j \in [s_i]$ let~$p_j \deq \phi(w_j)$, 
so~$p_j \in P_i$. Since~$W \subseteq V(T') \sm \{r\}$, for each~$j \in [s_i]$ the vertex~$p_j$ 
was embedded in~$V'_i \sm \{v\}$, so by Claim~\ref{c:reserve-XY}\ref{i:deg-X} 
we have~$\deg^0(p_j, X_i\setminus B_i) \geq \deg^0(p_j,X_i) - |B_i|= \eta m - s_i \geq s_i$. 
We may therefore choose distinct vertices~$x_1,\ldots, x_{s_i}$ in~$X_i\setminus B_i$ 
such that for each~$j\in[s_i]$,
the vertex~$x_j$ is an~inneighbour of~$p_j$ if~$b_j\in N^+(p_j)$, 
whilst~$x_j$ is an~outneighbour of~$p_j$ if~$b_j\in N^-(p_j)$.
For each~$j \in [s_i]$ let~$w_j^+$ be the removed out-leaf of~$T$ adjacent to~$w_j$ and 
let~$w_j^-$ be the removed in-leaf of~$T$ adjacent to~$w_j$. If~$b_j\in N^+(p_j)$ then 
we set~$\phi(w_j^+) = b_j$ and~$\phi(w_j^-) = x_j$, whilst if~$b_j\in N^-(p_j)$ then we 
set~$\phi(w_j^-) = b_j$ and~$\phi(w_j^+) = x_j$. Observe that our choice of vertices~$x_1, \dots, x_{s_i}$ 
ensures that these embeddings are consistent with the directions of the 
edges~$w_j^- \rarr w_j$ and~$w_j \rarr w_j^+$.

Having carried out these steps for each \ink\ we have extended the embedding~$\phi$ to cover
all vertices in~$B_1\cup\cdots\cup B_k$. For each \ink\ we now 
define~$W^0_i \deq W_i \sm \{w_1, \dots w_{s_i}\}$ and~$P^0_i \deq P_i \sm \{p_1, \dots, p_{s_i}\}$. 
In other words,~$W^0_i$ is the set of vertices of~$W$ 
which were embedded in~$V'_i$ and whose adjacent removed leaves have not yet been embedded, 
and~$P^0_i$ is the set of vertices of~$G$ to which vertices of~$W^0_i$ 
have been embedded. 
By~\eqref{e:wuniform} we then have 
\begin{equation}\label{e:Wi-after-step1}
|P^0_i| = |W_i^0| = |W_i| - s_i = \min_{i \in [k]} |W_i| - \lfloor 4\eps m\rfloor 
= \frac{t}{k} - 4\eps m \pm \delta n,
\end{equation}
so in particular we have~$|W^0_1| = \dots = |W^0_k| = |P^0_1| = \dots = |P^0_k|$. 
Similarly, for each \ink\ we define~$U^0_i \deq U_i \sm \{b_1,x_1,\ldots,b_{s_i},x_{s_i}\}$. 
In other words,~$U^0_i$ is the set of vertices of~$V_i$ which have not yet been covered by~$\phi$. 
By \eqref{e:uuniform} we then have
\begin{equation}\label{e:Ui-after-step1}
|U^0_i| = |U_i| - 2s_i = \frac{2t}{k} - 8\eps m \pm 6\delta n.
\end{equation}
Write~$W^0 \deq \bigcup_\ink W_i^0$,~$P^0 \deq \bigcup_\ink P^0_i$, and~$U^0 \deq \bigcup_\ink U_i^0$.
So in particular~$U^0$ is the set of vertices of~$G$ which remain uncovered. 
Since there are two such vertices for each vertex of~$W^0$, and~$|W^0_1| = |W^0_2| = \dots = |W^0_k|$, 
it follows that~$|U^0|$ is divisible by~$2k$.

\medskip\emph{Step 2: Balancing the numbers of uncovered vertices.} 
Our next step is to embed the removed leaves adjacent to a small number of vertices of~$W$ so that, 
following these embeddings, there are equally many uncovered vertices within each cluster (we also
preserve the properties ensured in Step 1).
\begin{figure}[tb] 
  \begin{center}
    \includegraphics{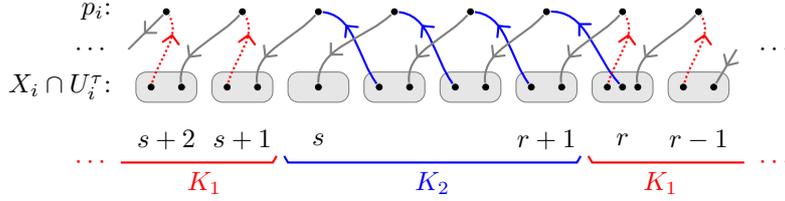}

    \caption{%
    This diagram illustrates how the embedding~$\phi$ is extended at each step of the balancing algorithm.
    The vertices at the top are the vertices~$p_1,\ldots,p_k$, which lie in the 
    sets~$P_1^\tau,\ldots,P_k^\tau$
	respectively,
	and the shaded areas represent the sets~$X_i \cap U_i^\tau$ for~$i \in [k]$ (that is, the vertices
	of~$X_i$ not yet covered by~$\phi$).
	The extension of~$\phi$ at step~$\tau$ then covers the vertices appearing in the shaded areas, so
	three extra vertices are covered from~$V_r$, one from~$V_s$, and two from each other cluster. 
}
    \label{fi:balancing}
  \end{center}
\end{figure} 
We achieve this by applying the following `balancing algorithm'. 
Each iteration of this algorithm will 
extend~$\phi$ by embedding, for each~\ink, 
the removed in-leaf and out-leaf adjacent to some vertex in~$W_i$.

More precisely, the balancing algorithm proceeds as follows.
For each time~$\tau \geq 0$
and for each~\ink,  
we let~$W_i^\tau \subseteq W_i$ be the set of vertices of~$T$ 
whose adjacent removed leaves have not yet been embedded, 
we let~$P_i^\tau \subseteq P_i$ be the set of vertices of~$G$ to which vertices of~$W_i^\tau$ have been
embedded, and
we let~$U_i^\tau \subseteq V_i$ be the set of uncovered vertices in~$V_i$ at~time~$\tau$.
Observe that these definitions of~$W_i^0$,~$P_i^0$ and~$U_i^0$ coincide with those given above.
We~also define the quantity~$M^\tau \deq \frac{1}{k} \sum_{i\in [k]}|U_i^\tau|$, so~$M^\tau$ is the
average number of uncovered vertices per cluster at time~$\tau$. Our observation above 
that~$|U^0|$ is divisible by~$2k$ ensures that~$M^0$ is an even integer, and in fact the algorithm will 
ensure that~$M^\tau$ is an even integer at each time~$\tau \geq 0$.
At time step~$\tau$, if~$|\,U_i^\tau\,| = M^\tau$ for all~\ink, then we stop with success.
Otherwise, since~$M^\tau$ is an integer, we may choose~$r, s \in [k]$ with 
$|U_r^\tau| \geq M^\tau + 1$ and~$|U_s^\tau| \leq M^\tau - 1$.
Define~$K_1 \deq  \{s+1,s+2,\ldots,r-1,r\}$ and~$K_2\deq \{r+1,\ldots,s\} = [k] \setminus K_1$,
with addition taken modulo~$k$.
For each~\ink, we choose a vertex~$w_i \in W_i^\tau$, and let~$p_i \in P_i^\tau$ be the vertex
to which~$w_i$ was embedded. We also choose a vertex~$x_i^+\in N^+(p_i)\cap X_\ipp \cap U_\ipp^\tau$ and, 
if~$i\in K_1$ 
then we choose a vertex~$x_i^-\in N^-(p_i)\cap X_i \cap U_i^\tau$, whilst if~$i \in K_2$
then we choose a vertex~$x_i^-\in N^-(p_i)\cap X_\imm \cap U_\imm^\tau$. 
We make these choices so that the~$2k$ vertices 
$\{x_1^-,x_1^+,\ldots,x_k^-,x_k^+\}$ are all distinct 
(if it is not possible to make such choices then we terminate with failure, but we shall see shortly
that this will not happen).
For each~$i \in [k]$ let~$w_i^+$ be the removed out-leaf of~$T$ adjacent to~$w_i$ and 
let~$w_i^-$ be the removed in-leaf of~$T$ adjacent to~$w_i$; we 
then set~$\phi(w_i^-) \deq x_i^-$ and~$\phi(w_i^+) \deq x_i^+$  
(see Figure~\ref{fi:balancing} for an illustration of this embedding).
To conclude this iteration of the algorithm, for each \ink\ we update the sets~$W_i^\tau$,~$P_i^\tau$ 
and~$U_i^\tau$ by 
setting~$W_i^{\tau+1} \deq W_i^\tau \sm \bigcup_\ink \{w_i\}$,~$P_i^{\tau+1}\deq P_i^\tau \sm \bigcup_\ink \{p_i\}$, 
and~$U_i^{\tau+1}\deq U_i^\tau \sm \bigcup_\ink \{x_i^+, x_i^-\}$. Observe that we then have 
\begin{equation} \label{e:updateu}
|U_i^{\tau+1}| = 
\begin{cases}
|U_i^{\tau}|-3 & \mbox{if~$i = r$,}\\
|U_i^{\tau}|-1 & \mbox{if~$i = s$, and}\\
|U_i^{\tau}|-2 & \mbox{otherwise.}
\end{cases}
\end{equation}
In particular it follows that~$M^{\tau+1} = M^\tau -2$; since~$M^\tau$ was an even integer 
it follows that~$M^{\tau+1}$ is an even integer, as required.

\begin{claim}
The balancing algorithm described above stops with~success after at~most~$3k\delta n$~iterations.
\end{claim}

\begin{claimproof}
We first check that we can choose vertices~$w_i, p_i, x_i^-$ and~$x_i^+$ as described whenever~$\tau \leq 3k \delta n$. 
First observe that, for each~$i \in[k]$, since~$|W_i^0| \geq t/2k > 3k\delta n$ by~\eqref{e:Wi-after-step1}, 
and at most one vertex is removed from~$W_i^\tau$ at each step~$\tau$ of the balancing algorithm 
(and its image is removed from~$P_i^\tau$), there are at least~$|W_i^0| - \tau \geq 1$ possible 
choices for~$w_i$ at step~$\tau \leq 3k\delta n$ of the balancing algorithm. So we may choose the 
vertices~$w_i$ and~$p_i$ for~$i \in [k]$ as claimed. Next observe that for each~$i \in [k]$ 
at most~$2s_i \leq 8\eps m + 4\delta n$ vertices were embedded in~$X_i$ in Step 1. Also, each iteration 
of the balancing algorithm embeds at most three vertices in~$X_i$, so at time~$\tau \leq 3k \delta n$ the 
total number of vertices which have so far been embedded in~$X_i$ is 
at most~$3 \tau + 8 \eps m + 4\delta n \leq 9k \delta n + 9\eps m \leq \eta m/2$. 
Since~$p_i \in V'_i$, it follows by Claim~\ref{c:reserve-XY}\ref{i:deg-X} and~\ref{i:deg-Y} 
that~$\deg^-(p_i, X_\imm \cap U_\imm^\tau) \geq \eta m/2$, 
that~$\deg^+(p_i, X_\ipp \cap U_\ipp^\tau) \geq \eta m/2$ 
and that~$\deg^-(p_i, X_i \cap U_i^\tau) \geq \eta m/2$. 
So we may greedily choose the vertices~$x_i^-$ and~$x_i^+$ for each~$i \in [k]$ as desired.

It therefore suffices to prove that the algorithm stops after
at most~$3k\delta n$~iterations and thus,
because it cannot fail in these early steps, it always stops
successfully. For each~$\tau \geq 0$ let~$\Upsilon^\tau \deq \sum_{\ink} \big||\,U_i^\tau\,| - M^\tau\big|$, 
so~$\Upsilon^\tau$ is a non-negative integer.
In particular by~\eqref{e:Ui-after-step1} we have
\begin{equation}\label{e:upsilon0}
\Upsilon^0
= \sum_{\ink} \big||\,U_i^0\,| - M^0 \big|
\leq 6k\delta n,
\end{equation}
Also, by~\eqref{e:updateu} we have~$|U_r^{\tau+1}| = |U_r^{\tau}| - 3$ and~$M^{\tau+1} = M^\tau -2$; 
by our choice of~$r$ it follows 
that~$\big||\,U_r^{\tau+1}\,| - M^{\tau+1}\big| = \big||\,U_r^{\tau}\,| - M^{\tau}\big| - 1$. 
Similarly we find 
that~$\big||\,U_s^{\tau+1}\,| - M^{\tau+1}\big| = \big||\,U_s^{\tau}\,| - M^{\tau}\big| - 1$ and 
that~$\big||\,U_j^{\tau+1}\,| - M^{\tau+1}\big| = \big||\,U_j^{\tau}\,| - M^{\tau}\big|$ 
for each~$j \in [k] \sm \{r, s\}$. Together these equalities imply 
that~$\Upsilon^{\tau+1} = \Upsilon^\tau - 2$.
Since~$\Upsilon^\tau$ is always non-negative, we conclude that for some~$\tau \leq 3k\delta n$ we must 
have~$\Upsilon^\tau = 0$. It follows that~$|U_j^\tau| = M^\tau$ for all~$j\in[k]$, and so the algorithm 
will stop at step~$\tau$.
\end{claimproof}

Returning to the~proof of~Lemma~\ref{l:2nd-half}, we conclude that the balancing algorithm 
will stop with success at some time~$\tauend$ with~$\tauend\leq 3k\delta n$.
For each~\ink, let~$W_i^* \deq W_i^\tauend$,~$P^*_i \deq P_i^\tauend$, and~$U^*_i \deq U_i^\tauend$, 
and write~$W^* \deq \bigcup_\ink W_i^*$,~$P^* \deq \bigcup_\ink P^*_i$, and~$U^* \deq \bigcup_\ink U_i^*$. 
So the embedding~$\phi$ now covers all vertices of~$V(G)$ except for those in~$U^*$, 
and the only vertices of~$T$ which remain to be embedded are one in-leaf and one out-leaf of 
each vertex of~$W^*$. In particular we have~$|U^*| = 2|P^*| = 2|W^*|$.
Observe that in the execution of the balancing algorithm,
at each time~$\tau$ and for each~\ink\ 
precisely one vertex was removed from~$W_i^\tau$.
Therefore, since we initially had~$|W_1^0| = \cdots = |W_k^0|$ by~\eqref{e:Wi-after-step1}, 
we now have~$|W^*_1| = \cdots = |W^*_k|$. We denote this common size by~$L$, 
and note that by~\eqref{e:Wi-after-step1} we then have~$L \geq t/k - 4\eps m - \delta n - \tauend \geq 2t/3k$. 
Also, since~$\Upsilon^\tauend = 0$, we must have~$|U_1^*| = \dots = |U_k^*| = M^\tauend$, so 
 \begin{equation}\label{e:final-sets-are-large}
L = |W^*_1| = \cdots = |W^*_k| = |P^*_1| = \cdots = |P^*_k| = \frac{1}{2}|U^*_1| = \cdots = 
\frac{1}{2}|U^*_k| \geq \frac{2t}{3k} \geq \frac{\beta m}{2}.
\end{equation}

\medskip\emph{Step 3: Completing the embedding.} 
We are now ready to complete the embedding of~$T$ in~$G$ as described previously, 
beginning with the following claim.

\begin{claim}\label{c:P-U-super-regular}
For each~\ink\ each vertex in~$U^*_i$ has at least~$\eta m/2$ inneighbours in~$P^*_\imm$ and at 
least~$\eta m/2$ outneighbours in~$P^*_\ipp$, 
and each vertex in~$P^*_i$ has at least~$\eta m$ inneighbours in~$U^*_\imm$ and 
at least~$\eta m$ outneighbours in~$U^*_\ipp$.
\end{claim}

\begin{claimproof}
Recall that the set~$B_i$ chosen in Step 1 contained all vertices of~$U_i$ with fewer than~$\eta m$ 
inneighbours in~$P_\imm$ or fewer than~$\eta m$ outneighbours in~$P_\ipp$. All vertices of~$B_i$ were 
covered in Step 1, so no vertex of~$B_i$ is contained in~$U^*_i$. The first statement then follows from the 
fact that for each~$j \in [k]$ we have 
\[ |P_j \sm P_j^*| = |P_j \sm P_j^0| + |P_j^0 \sm P_j^*| \leq s_j + \tauend 
\leq 4\eps m + 2 \delta n + 3k\delta n \leq \frac{\eta m}{2}.\]
For the second statement observe that no vertices have yet been embedded in any set~$Y_j$, 
so~$Y_\imm \subseteq U^*_\imm$ and~$Y_\ipp \subseteq U^*_\ipp$. Moreover, 
since~$P^*_i \subseteq P_i \subseteq V_i$, by Claim~\ref{c:reserve-XY}\ref{i:deg-Y} every vertex 
of~$P^*_i$ has at least~$\eta m$ inneighbours in~$Y_\imm$ and at least~$\eta m$ outneighbours in~$Y_\ipp$.
\end{claimproof}

For each~$\ink$ we now partition~$U^*_i$ into disjoint sets~$U_i^-$ and~$U_i^+$ each of size~$L$ 
uniformly at random and independently of all other choices. 
Since~$G[V_i\rarr V_\ipp]$ is~$(\dplus,\eps)$-regular for each~\ink, 
by~\eqref{e:final-sets-are-large} and Lemma~\ref{l:slice-pair} both 
$G[U_\imm^- \rarr P^*_i]$ and~$G[P^*_i \rarr U_\ipp^+]$ are then~$(\dplus,\eps')$-regular, where~$\eps' \deq 3\eps/\beta$.
Also, by Claim~\ref{c:P-U-super-regular}, 
each~$u \in U_\imm^-$ has~$\deg^+(u, P^*_i) \geq \eta m/2 \geq \eta L/2$ and
each~$u \in U_\ipp^+$ has~$\deg^-(u, P^*_i) \geq \eta m/2 \geq \eta L/2$.
Furthermore, for each~$p \in P^*_i$ the 
random variables~$\deg^-(p, U_\imm^-)$ and~$\deg^+(p, U_\ipp^+)$ each have 
hypergeometric distributions with expectation at least~$\eta m L/2L \geq \eta L/2$. 
Applying Theorem~\ref{t:exp} and taking a union bound we find that with 
positive probability we have for every~\ink\ and every~$p \in P^*_i$ 
that~$\deg^-(p, U_\imm^-) \geq \eta L/4$ and~$\deg^+(p, U_\ipp^+) \geq \eta L/4$. 
Fix such an outcome of our random selection; then for each~\ink\ the 
underlying graphs of both~$G[U_\imm^- \rarr P^*_i]$ and~$G[P^*_i \rarr U_\ipp^+]$ 
are~$(\eta/4, \eps')$-super-regular balanced bipartite graphs with 
vertex classes of size~$L$.

We may therefore apply Lemma~\ref{l:matching} to obtain, for each~\ink, a perfect 
matching~$M_i^-$ in~$G[U_\imm^- \rarr P^*_i]$ and a perfect matching~$M_i^+$ in~$G[P^*_i \rarr U_\ipp^+]$. 
For each~\ink\ and each~$w \in W^*_i$ 
let~$w^-$ be the removed in-leaf of~$T$ adjacent to~$w$ and
let~$w^+$ be the removed out-leaf of~$T$ adjacent to~$w$. 
Also let~$p = \phi(w)$ and let~$q^- \in U_\imm^-$ and~$q^+ \in U_\ipp^+$ be the vertices 
matched to~$p$ in~$M_i^-$ and~$M_i^+$ respectively, and set~$\phi(w^-) \deq q^-$ and~$\phi(w^+) \deq q^+$.
Since each~$p \in P^*_i$ is matched to precisely one inneighbour in~$U_\imm^-$ and precisely one 
outneighbour in~$U_\ipp^+$, this extends~$\phi$ to an embedding of~$T$ in~$G$.
\end{proof}

\subsection{Joining the pieces}\label{s:joining}

As outlined at the start of this section, we will `split' our tree~$T$ into two 
subtrees~$T_1$ and~$T_2$, which we embed successively in~$G$ using Lemmas~\ref{l:1st-half} 
and~\ref{l:2nd-half}. Definition~\ref{d:tree-partition} makes this notion precise, 
following which Lemma~\ref{l:split-tree} shows that every oriented tree admits such a split.

\begin{definition}\label{d:tree-partition}
Let~$T$ be a tree or oriented tree. A~\defi{tree-partition} of~$T$ is a pair~$\{T_1,T_2\}$ of 
edge-disjoint subtrees of~$T$ such that~$V(T_1) \cup V(T_2) = V(T)$ and~$E(T_1) \cup E(T_2) = E(T)$ 
(so in particular~$T_1$ and~$T_2$ have precisely one vertex in common).
\end{definition}

\begin{lemma}\label{l:split-tree}
Let~$T$ be a tree or oriented tree. For every set~$L\subseteq V(T)$ there exists
a~tree-partition~$\{T_1, T_2\}$ of~$T$ such that~$T_1$ and~$T_2$ each contain
at~least~$|L|/3$ vertices of~$L$.
\end{lemma}

We prove Lemma~\ref{l:split-tree} using the following simple fact.

\begin{fact}\label{f:sum}
Let~$x_1,\ldots,x_s$ be non-negative integers.
If~$x_1,\ldots,x_s \leq c$ and~$x_1+\cdots+x_s\geq 3c$, then there exists~$i\in[s]$ such that
$c\leq x_1+\cdots+x_i\leq 2c$.
\end{fact}

For a tree or oriented tree $T$, an edge~$e\in E(T)$ and a vertex~$v\in e$,
we write~$T-e$ for the oriented forest we obtain by deleting~$e$ from~$T$, 
and write~$C_v^e$ for the vertex set of the component of~$T-e$ which contains~$v$.

\begin{proof}[Proof of~Lemma~\ref{l:split-tree}]
Since edge orientations do not affect the validity of a tree-partition, 
we may assume that $T$ is an undirected tree.
Define~$\ell\deq |L|$. For each edge~$e=\{u,v\}\in E(T)$ we say 
that~$v$ is a~\defi{heavy neighbour} of~$u$ if~$|C_v^e \cap L| \geq \ell/3$.
Observe that if~$u$ and~$v$ are both heavy neighbours of each other, then~$\bigl\{T\bigl[C_u^e\bigr], T\bigl[C_v^e \cup \{u\}\bigr]\bigr\}$ 
is the desired tree-partition. 
We may therefore assume that for each edge~$e=\{u,v\}\in E(T)$ either~$u$ is a heavy neighbour of~$v$ 
or~$v$ is a heavy neighbour of~$u$, but not both.
It follows that some vertex~$v$ has no heavy neighbours 
(to see this, form an auxiliary orientation of~$E(T)$ with each edge directed~$u \rarr v$ where~$v$ is a 
heavy neighbour of~$u$, and choose~$v$ to be a vertex with no outneighbours).
Let~$C_1,\ldots,C_s$ be the vertex sets of the components of~$T-v$.
For each~$i\in[s]$ let~$\ell_i\deq |C_i\cap L|$, and observe that since~$v$ has no heavy neighbours we
have~$\ell_i < \ell/3$;
since~$\ell_i$ is an~integer, we then have~$\ell_i\leq (\ell -1)/3$.

If~$v\in L$, then~$\ell_1+\cdots+\ell_s=\ell -1$
and by~Fact~\ref{f:sum} there exists~$j\in [s]$ with
$(\ell -1)/3\leq \ell_1+\cdots+\ell_j\leq 2(\ell-1)/3$.
In this case the desired tree-partition is~$
\bigl\{T\bigl[\{v\}\cup\bigcup_{1\leq i\leq j} C_i\bigr],\, 
T\bigl[\{v\}\cup\bigcup_{j < i\leq s} C_i\bigr]\bigr\}$,
because each of these subtrees contains~$v$ and hence contains
 at~least~$(\ell -1)/3 + 1 > \ell/3$ vertices of~$L$.
On the other hand, if~$v\notin L$, 
then~$\ell_1+\cdots+\ell_s=\ell$,
so by Fact~\ref{f:sum} above there exist~$j\in [s]$ with
$\ell/3\leq \ell_1+\cdots+\ell_j \leq 2\ell/3$, and again 
$
\bigl\{T\bigl[\{v\}\cup\bigcup_{1\leq i\leq j} C_i\bigr],\,
T\bigl[\{v\}\cup\bigcup_{j< i\leq s} C_i\bigr]\bigr\}$
is the desired tree-partition.
\end{proof}

We are now ready to state and prove Lemma~\ref{l:in-CCT}, the main result of this section.

\begin{lemma}\label{l:in-CCT}
Suppose that~$1/n \ll 1/C$ and that~$1/n \ll 1/k \ll \eps\ll d\ll\psi\ll\alpha.$
Let~$T$ be an~$\alpha$-nice oriented tree on~$n$~vertices with 
maximum degree~$\Delta(T)\leq (\log n)^C$.
Also let~$G$ be a tournament on~$n$ vertices which contains a~$(d, \eps)$-regular \CCT\ 
whose clusters~$V_1, \dots, V_k$ have equal size such that~$B \deq V(G) \sm \bigcup_\ink V_i$ has 
size~$|B| \leq \psi n$. Then~$G$ contains a (spanning) copy of~$T$.
\end{lemma}

\begin{proof}[Proof of~Lemma~\ref{l:in-CCT}]
Introduce a new constant~$\beta$ with 
$\psi \ll \beta \ll \alpha$,
and define~$m \deq |V_1|=\cdots=|V_k| = \bigl(n-|B|\bigr)/k$ and~$s \deq \lceil \alpha n \rceil$. 
Since~$T$ is~$\alpha$-nice we may choose a set~$L$ of~$s$~distinct vertices of~$T$ such that
each vertex in~$L$ is adjacent to at least one in-leaf and at least one out-leaf of~$T$.
Apply Lemma~\ref{l:split-tree} to obtain a tree-partition~$\{T_1, T_2\}$ of~$T$ 
such that the subtrees~$T_1$ and~$T_2$ each contain at least~$s/3$ vertices of~$L$. 
Let~$r$ be the unique common vertex of~$T_1$ and~$T_2$, which we take as the root of each subtree, 
and observe that for each vertex~$x \neq r$ every neighbour of~$x$ is contained
in the same subtree as~$x$. 
So in particular~$T_1$ contains at least~$s/3-1 \geq \alpha n/4 \geq \beta n$ vertices each adjacent
to at least one in-leaf and at least one out-leaf of~$T_1$, 
and likewise~$T_2$ contains at least~$\alpha n/4 \geq \beta n$ vertices each adjacent
to at least one in-leaf and at least one out-leaf of~$T_2$. 
Write~$n_1 \deq |T_1|$ and~$n_2 \deq |T_2|$, so~$3 \alpha n/4 \leq n_1, n_2$ and~$n_1 + n_2 = n+1$.
By relabelling if necessary we may assume that~$n_1 \leq n_2$. Observe also that
$\Delta(T_1) \leq \Delta(T) \leq (\log n)^C \leq (\log n_1)^{2C}$ and likewise that
$\Delta(T_2) \leq (\log n_2)^{2C}$. So~$T_1$ meets the conditions of
Lemma~\ref{l:1st-half} with~$2C$ and~$n_1$ in place of~$C$ and~$n$ respectively, 
and likewise~$T_2$ meets the conditions of
Lemma~\ref{l:2nd-half} with~$2C$ and~$n_2$ in place of~$C$ and~$n$ respectively.

Next, proceed as follows for each~\ink. 
Define~$B_i^+ \deq\bigl\{\, v\in V_i: \deg^+(v, V_\ipp) < (d - \eps)m\,\bigr\}$,
and~$B_i^- \deq\bigl\{\, v\in V_i: \deg^-(v, V_\imm) < (d - \eps)m\,\bigr\}$.
Since~$G[V_\imm \rarr V_i]$ and~$G[V_i \rarr V_\ipp]$ are each~$(\dplus, \eps)$-regular, we must then 
have~$|B_i^-|, |B_i^+| < \eps m$. 
Let~$B_i$ be a set of~$2\eps m$ vertices 
such that~$B_i^-\cup B_i^+\subseteq B_i\subseteq V_i$ and define~$V_i'\deq V_i\setminus B_i$.
It follows that for every vertex~$x \in V'_i$ 
we have~$\deg^-(v, V'_\imm), \deg^+(v, V'_\ipp) \geq (d-\eps) m - 2\eps m = (d - 3\eps) m$.
Choose a subset~$Z_i \subseteq V'_i$ of size~$|Z_i| = \alpha m/5$ uniformly at random 
and independently of all other choices. So for each~$x \in V'_i$ the random variables 
$\deg^-(x, Z_\imm)$ and~$\deg^+(x, Z_\ipp)$ have hypergeometric distributions with expectation at least
$(d - 3\eps)\alpha m/5$. Applying~Theorem~\ref{t:exp} and 
taking a union bound we find that with positive probability 
we have for every~$x \in V'_i$ that~$\deg^-(x, Z_\imm), \deg^+(x, Z_\ipp) \geq d \alpha m/10$. Fix an outcome 
of the random selections for which this event occurs.

Define~$B' \deq B \cup \bigcup_\ink B_i$, so~$|B'| = |B| + 2k\eps m \leq 2\psi n$. 
Next choose arbitrarily a set~$X_i \subseteq V'_i \sm Z_i$ of size~$(1+\alpha/4)n_1/k$ for each~\ink;
this is possible since for each~\ink\ we have 
\[|V'_i \sm Z_i| = (1-2\eps)m - \frac{\alpha m}{5} \geq \left(1-\frac{\alpha}{4}\right)m 
\geq \left(1-\frac{\alpha}{4}\right)(1-\psi)\frac{n}{k} \geq \left(1-\frac{\alpha}{3}\right)\frac{n}{k} 
\geq \left(1+\frac{\alpha}{3}\right)\frac{n_1}{k},\]
where the final inequality uses the fact that~$n_1 \leq n+1-n_2 \leq n+1 - 3\alpha n/4 \leq (1-2\alpha/3)n$.
Define~$G_1 \deq G\bigl[B' \cup \bigcup_\ink X_i\bigr]$.
Since~$G[V_i \to V_\ipp]$ is~$(\dplus, \eps)$-regular for each~\ink, and~$n_1 \geq 3\alpha n/4$, 
it follows by Lemma~\ref{l:slice-pair} that the sets~$X_1, \dots, X_k$ are 
the clusters of a~$(d, \eps')$-regular \CCT\ in~$G_1$, where~$\eps' \deq 4\eps/3\alpha$. 
The tournament~$G_1$, the clusters~$X_i$ and the set~$B'$ therefore
meet the conditions of Lemma~\ref{l:1st-half} with~$n_1, \alpha/3, \eps'$ and~$2\psi$ in 
place of~$n, \alpha, \eps$ and~$\psi$ respectively. 
So we may apply Lemma~\ref{l:1st-half} to obtain an embedding~$\phi$ of~$T_1$ in~$G_1$ so 
that~$r$ is embedded in~$X_1$, so that every vertex of~$B'$ is covered, and so that for each~$i \in [k]$ we have 
\begin{equation}\label{e:phiimage}
\bigl|\phi\bigl(V(T)\bigr) \cap X_i\bigr| = \bigl(n_1-|B'|\bigr)\left(\frac{1}{k} \pm \frac{2}{\log \log n_1}\right) = 
\frac{n_1-|B'|}{k} \pm \left(\frac{2n_1}{\log \log n_1}-2\right).
\end{equation}

For convenience of notation write~$E \deq \frac{2n_2}{\log \log n_2} \geq \frac{2n_1}{\log \log n_1}$.
For each~\ink\ define~$U_i \deq V_i \sm \phi\bigl(V(T)\bigr)$, so~$U_i$ contains all vertices of~$V_i$ not covered by 
our embedding of~$T_1$. 
Then by~\eqref{e:phiimage} we  
have for each~\ink\ that 
\begin{align*}
|U_i| &= \bigl|V_i \sm B_i\bigr| - \frac{n_1-|B'|}{k} \pm (E-2) = 
m - 2\eps m - \frac{n_1}{k} + \frac{|B| + 2k\eps m}{k} \pm (E-2)\\
& = \frac{n - |B|}{k} - \frac{n_1}{k} + \frac{|B|}{k} \pm (E-2) = \frac{n_2}{k} \pm (E-1),
\end{align*}
where the second equality uses the fact that~$|B'| = |B| + 2k\eps m$, 
and the final equality uses the fact that~$n_2 = n+1-n_1$. 
Let~$v = \phi(r)$, so~$v \in X_1$, and set~$U_1^* \deq U_1 \cup \{v\}$ and~$U_i^* \deq U_i$ for~$2 \leq i \leq k$, 
so~$|U_i^*| = \frac{n_2}{k} \pm E$ for each \ink.
In particular, we have~$|U^*_i| \geq \alpha n/2k \geq \alpha |V_i|/2$ for each~\ink, 
so by Lemma~\ref{l:slice-pair} the sets~$U^*_1, \dots, U^*_k$ are the 
clusters of a spanning~$(d, 2\eps/\alpha)$-regular \CCT\ in the 
tournament $G_2 \deq G[U^*_1 \cup \dots \cup U^*_k]$. 
Furthermore, for each~\ink\ we have~$Z_i \subseteq U^*_i \subseteq V'_i$ 
(since we chose~$X_i$ to be disjoint from~$Z_i$, and
every vertex of~$B$ was covered by the embedding of~$T_1$), 
so every vertex~$u \in U^*_i$ has 
$\deg^-(x, U^*_\imm), \deg^+(x, U^*_\ipp) \geq d\alpha m/10$. So in fact the 
clusters~$U^*_1, \dots, U^*_k$ form a spanning~$(d\alpha/10, 2\eps/\alpha)$-super-regular \CCT\ in~$G_2$.
In other words, the tournament~$G_2$, the clusters~$U^*_1, \dots, U^*_k$ and the vertex~$v$ meet 
the conditions of Lemma~\ref{l:2nd-half} with~$d\alpha/10, 2\eps/\alpha$ and~$n_2$ in 
place of~$d, \eps$ and~$n$ respectively. Since $|G_2| = |G| - |T_1| + 1 = n - n_1 +1 = n_2 = |T_2|$ we may therefore apply 
Lemma~\ref{l:2nd-half} to find a spanning copy of~$T_2$ in~$G_2$ in which~$r$ is 
embedded to~$v$, and then the 
embeddings of~$T_1$ and~$T_2$ together form a spanning copy of~$T$ in~$G$.
\end{proof}

\section{Proofs of main theorems}\label{s:main-thms}

In this section we give the proofs of Theorem~\ref{t:good-unavoidable} 
(that every large nice oriented tree of polylogarithmic maximum degree is unavoidable) and 
Theorem~\ref{t:most-are-nice} (that a random labelled oriented tree is nice asymptotically almost surely).

\subsection{A class of unavoidable oriented trees}

We begin by combining the results of the previous two sections to prove Theorem~\ref{t:good-unavoidable}. 
The main task is to use Lemma~\ref{l:struct} to show that we can find either 
an almost-directed pair in~$G$ which partitions~$V(G)$ or an almost-spanning \CCT\ in~$G$. 
In the former case we then embed~$T$ in~$G$ using Lemma~\ref{l:embed-nice-in-cut}, 
whilst in the latter case we embed~$T$ in~$G$ using~Lemma~\ref{l:in-CCT}.

\begin{proof}[Proof of Theorem~\ref{t:good-unavoidable}]
Introduce new constants~$k_0, k_1, \eps,d, \mu, \eta, \omega$ and~$\gamma$ such that
\[
\frac{1}{n} \ll \frac{1}{k_1} \ll \frac{1}{k_0} \ll \eps\ll d\ll\mu\ll\eta\ll\omega\ll\gamma\ll \alpha.
\]
We may also assume that~$1/n \ll 1/C$.
Let~$G$ be a tournament on~$n$ vertices,
and let~$T$ be an $\alpha$-nice tree on $n$ vertices such that~$\Delta(T)\leq (\log n)^C$.
We choose vertex-disjoint subsets~$X, Y, Z \subseteq V(G)$ such that
\begin{enumerate}[label=(\alph*)]
\item \label{i:1.4a} $X \cup Y \cup Z = V(G)$,
\item \label{i:1.4b} $|Y| \geq n/3$, and
\item \label{i:1.4c} $e\bigl(G(Y \rarr X)\bigr) + e\bigl(G(Z \rarr X)\bigr) + e\bigl(G(Z \rarr Y)\bigr) 
\leq \min\bigl(\eta \bigl(|X|+|Z|\bigr)n, 3\gamma \eta n^2\bigr)$.
\end{enumerate}
Moreover, we make this choice so that~$|Y|$ is minimal among all choices of~$X$,~$Y$ and~$Z$ which 
satisfy~\ref{i:1.4a}--\ref{i:1.4c} above 
(taking~$Y = V(G)$ and~$X = Z = \emptyset$ shows that such subsets do exist).

Suppose first that~$|Y| \leq (1- 2\gamma) n$. Then we have either~$|X| \geq \gamma n$ 
or~$|Z| \geq \gamma n$. If~$|X| \geq \gamma n$ then, taking~$A \deq X$ and~$B \deq Y \cup Z$, 
we have a partition~$\{A, B\}$ of~$V(G)$ into sets~$|A|, |B| \geq \gamma n$ such that the number of 
edges directed from~$B$ to~$A$ is 
$e\bigl(G(Y \rarr X)\bigr) + e\bigl(G(Z \rarr X)\bigr) \leq 3\gamma \eta n^2 \leq \omega |A||B|$ by~\ref{i:1.4c}, so~$(A, B)$ is 
an~$\omega$-almost-directed pair in~$G$. If instead~$|Z| \geq \gamma n$ then a similar argument shows 
that taking~$A \deq X \cup Y$ and~$B = Z$ gives a partition~$\{A, B\}$ of~$V(G)$ into 
sets~$|A|, |B| \geq \gamma n$ such that~$(A, B)$ is an~$\omega$-almost-directed pair in~$G$. Either way, 
we may then apply Lemma~\ref{l:embed-nice-in-cut} (with~$\omega$ and~$\gamma$ in place of~$\mu$ 
and~$\nu$ respectively) to~find a copy of~$T$ in~$G$.

Now suppose instead that~$|Y| > (1 - 2\gamma) n$, and write~$G' \deq G[Y]$. 
Observe in particular that we then have $|X| + |Z| = n - |Y| < 2 \gamma n$. 
If there exists a vertex~$y \in Y$ with~$\deg^-_{G'}(y) < \eta n$, then moving~$y$ from~$Y$ to~$X$ 
would increase~$e\bigl(G(Y \rarr X)\bigr)$ by less than~$\eta n$ whilst increasing~$|X|$ by one 
and leaving $e\bigl(G(Z \rarr X)\bigr) + e\bigl(G(Z \rarr Y)\bigr)$ and~$|Z|$ unchanged. 
The resulting sets would then satisfy~\ref{i:1.4a},~\ref{i:1.4b} and~\ref{i:1.4c} with 
a smaller value of~$|Y|$, contradicting the minimality of~$|Y|$ in our choice of~$X$,~$Y$ and~$Z$. 
So every vertex~$y \in Y$ must have~$\deg^-_{G'}(y) \geq \eta n$. Likewise, if there exists a 
vertex~$y \in Y$ with~$\deg^+_{G'}(y) < \eta n$, then we obtain a similar contradiction by moving~$y$ 
from~$Y$ to~$Z$. We conclude that every vertex~$y \in Y$ must have~$\deg^+_{G'}(y) \geq \eta n$, 
so~$\delta^0(G') \geq \eta n \geq \eta |Y|$. Now suppose that there exists a partition~$\{S, S'\}$ of~$Y$ 
such that~$(S, S')$ is a~$\mu$-almost-directed pair in~$G'$.
Observe that moving all vertices of~$S$ from~$Y$ to~$X$ would increase~$e(Y \rarr X)$ 
by at most~$e(S' \to S) \leq \mu |S||S'| \leq \gamma \eta |S| n$ whilst increasing~$|X|$ by~$|S|$ and 
leaving~$e\bigl(G(Z \rarr X)\bigr) + e\bigl(G(Z \rarr Y)\bigr)$ and~$|Z|$ unchanged. So if~$|S| \leq n/2$, then at least~$|Y| - n/2 \geq n/3$ 
vertices would remain in~$Y$, and so the resulting sets would satisfy~\ref{i:1.4a},~\ref{i:1.4b} 
and~\ref{i:1.4c} with a 
smaller value of~$|Y|$, again contradicting the minimality of~$|Y|$.  On the other hand, if~$|S| > n/2$ 
then~$|S'| \leq n/2$, and we obtain a similar contradiction by moving all vertices of~$S'$ from~$Y$ to~$Z$. 
We conclude that no such partition~$\{S, S'\}$ of~$Y$ exists. Therefore by Theorem~\ref{l:struct} there is 
an integer~$k$ with~$k_0 \leq k \leq k_1$ such that~$G'$ contains a~$(d, \eps)$-regular \CCT\ with 
clusters~$V_1, \dots, V_k$ of equal size such that~$|\bigcup_{\ink} V_i| > (1-\eps) |Y| \geq (1-3\gamma)n$. 
We may therefore apply Lemma~\ref{l:in-CCT} (with~$3\gamma$ in place of~$\psi$) to obtain 
a copy of~$T$ in~$G$.
\end{proof}

\subsection{Most oriented trees are nice}

We now turn to the proof of Theorem~\ref{t:most-are-nice}, for which we use the following 
well-known result, known as \Cayley's theorem.

\begin{theorem}[Borchardt 1860; \Cayley\ 1889]\label{t:cayley}
  There are~$n^{n-2}$ labelled undirected trees with vertex set~$[n]$.
\end{theorem}

A \defi{cherry} is a path of length two, and 
and its \defi{centre} is the vertex of degree two. 
In an oriented tree~$T$ we refer to an in-subtree (respectively out-subtree) 
which is an (oriented) cherry as an~\defi{in-cherry} (respectively~\defi{out-cherry}).
Our next lemma states that most labelled undirected trees have many pendant cherries. 
This is a special case of a much more general result for simply generated trees due to 
\Janson~\cite{janson2012simply}. For completeness, we include a proof of the particular
statement that suffices for our purposes.

\begin{lemma}\label{l:aas-many-cherries}
  Fix~$\eps > 0$, and let~$T$ be a~tree chosen uniformly at random from the set of all
  labelled undirected trees with vertex set~$[n]$.
  Then asymptotically almost surely~$T$ contains~$(1 \pm \eps) \frac{\ee^{-3}}{2}n$ pendant cherries.
\end{lemma}

\begin{proof}
For each set~$S \in \binom{[n]}{3}$, let~\vcher{S} be the indicator random variable
which has value~1 if~$S$ spans a~pendant cherry in~$T$ and~0 otherwise.
We first note that
\begin{equation}\label{e:Pabc}
\bbp(\,\vcher{S}=1\,)=\frac{3(n-3)(n-3)^{n-5}}{n^{n-2}}
= \frac{3}{n^2}\left(1-\frac{3}{n}\right)^{n -4}. \nonumber
\end{equation}
Indeed, there are three possible choices for the centre of the cherry, 
this centre is adjacent to one of the~$n-3$ vertices in~$[n] \setminus S$,
and by Theorem~\ref{t:cayley} there are~$(n-3)^{n-5}$ distinct possibilities for the undirected labelled 
tree spanned by~$[n] \sm S$, giving the numerator, whilst the denominator is simply the total number of 
labelled undirected trees on~$n$ vertices (again by Theorem~\ref{t:cayley}).
The number of pendant cherries in~$T$ is~$X \deq \sum_{S \in \binom{[n]}{3}} \vcher{S}$, so
by linearity of expectation it follows that
\begin{equation}\label{e:EX}
\bbe (X) 
 = \sum_{S \in \binom{[n]}{3}} \bbp(\vcher{S}=1)
 = \binom{n}{3}\frac{3}{n^2}\left(1 - \frac{3}{n}\right)^{n-4} = \bigl(1+\littleo(1)\bigr) \frac{\ee^{-3}}{2}n.
\end{equation}
It therefore suffices to show that~$X$ is concentrated around~$\bbe(X)$.
Consider any distinct~$S, S' \in \binom{[n]}{3}$, and note that if~$S$ intersects~$S'$ then we must have
$\vcher{S}\cdot\vcher{S'} = 0$. On the other hand, if~$S$ and~$S'$ are 
disjoint then by a similar argument as above we have
\begin{equation}\label{e:abcdef}
\bbe\bigl(\vcher{S}\cdot\vcher{S'}\bigr)
  = \bbp (\,\vcher{S}=\vcher{S'} = 1\,) 
   = \frac{\bigl[3(n-6)\bigr]^2(n-6)^{n-8}}{n^{n-2}}
   = \frac{9}{n^4}\left(1 -\frac{6}{n}\right)^{n-6},\nonumber
\end{equation}
so
\begin{align}
\bbe(X^2)\nonumber
 & = \bbe\Bigl( \sum_{S \in \binom{[n]}{3}}\vcher{S}^2 + 
\sum_{\substack{S, S' \in \binom{[n]}{3}\\S\neq S'}}\vcher{S}\cdot\vcher{S'}\Bigr) 
 = \sum_{S \in \binom{[n]}{3}}\bbe(\vcher{S})   + 
\sum_{\substack{S, S' \in \binom{[n]}{3}\\S \cap S' = \emptyset}}
\bbe\bigl(\,\vcher{S}\cdot\vcher{S'}\,\bigr)\\
& = \binom{n}{3}\frac{3}{n^2}\left(1-\frac{3}{n}\right)^{n-4} 
                                    + \binom{n}{3}\binom{n-3}{3}\frac{9}{n^4}\left(1-\frac{6}{n}\right)^{n-6}.
\label{e:EX2}
\end{align}
Combining
\eqref{e:EX} 
and~\eqref{e:EX2} we find that
\begin{align} 
\Var(X) \nonumber
 & = \binom{n}{3}\frac{3}{n^2}\left(1-\frac{3}{n}\right)^{n-4} 
                                    + \binom{n}{3}\binom{n-3}{3}\frac{9}{n^4}\left(1-\frac{6}{n}\right)^{n-6}
     - \left[ \binom{n}{3}\frac{3}{n^2}\left(1-\frac{3}{n}\right)^{n-4} \right]^2 \\
& = \bigl(1+\littleo(1)\bigr)\frac{\ee^{-3}}{2}n + \bigl(1+\littleo(1)\bigr)\frac{\ee^{-6}}{4}n^2 - 
\left(\bigl(1+\littleo(1)\bigr)\frac{\ee^{-3}}{2}n\right)^2 = \littleo(n^2).\label{e:varX}
\end{align}
By \Chebyshev's inequality,~\eqref{e:EX} and~\eqref{e:varX} it follows that
\[
\bbp \left(\, \bigl|X - \bbe(X)\bigr| > \frac{\eps}{2} \cdot \bbe(X)\right) 
\leq \frac{\Var(X)}{\bigl(\eps \bbe(X)/2\bigr)^2}  = \littleo(1),
\]
which together with~\eqref{e:EX} proves the lemma.
\end{proof}

We are now ready to prove Theorem~\ref{t:most-are-nice}, that almost all labelled 
oriented trees are~$\frac{1}{250}$-nice.

\begin{proof}[Proof of Theorem~\ref{t:most-are-nice}]
  Let~$\calt_n$ be the set of all labelled oriented trees with vertex set~$[n]$.
  Note that we can select an oriented tree~$T$ uniformly at random from~$\calt_n$
  using the following two-step random procedure:
  first select a~tree~$T_0$ uniformly at~random from the set of
  all labelled undirected trees with vertex set~$[n]$,
  then form a labelled oriented tree~$T$ by orienting each edge~$e$~of~$T_0$
  uniformly at~random
  and independently of all other choices.
  Indeed, since there are~$n^{n-2}$ possibilities for~$T_0$ by Theorem~\ref{t:cayley}, and 
  every tree in~$\calt_n$ has~$n-1$ edges,
  the probability that a given labelled oriented tree~$T$ is selected by this 
  two-step procedure is~$n^{2-n}2^{1-n}$.

Let~$C$ be the number of pendant cherries of~$T_0$, 
let~$X$ be the number of pendant in-cherries of~$T$ which contain an out-leaf of~$T$,
and let~$Y$ be the number of pendant out-cherries of~$T$ which contain both an in-leaf and out-leaf of~$T$.
Observe that the probability that a fixed pendant cherry of~$T_0$ contributes to~$X$ is~$3/8$, and 
likewise the probability that a fixed pendant cherry of~$T_0$ contributes to~$Y$ is~$1/4$. 
So~$X \sim \calb(C, 3/8)$ and~$Y \sim \calb(C, 1/4)$. Since by Lemma~\ref{l:aas-many-cherries} we 
have~$C \geq n/50$ asymptotically almost surely (where we use the fact that~$\ee^{-3}/2 > 1/50$), 
it follows by Theorem~\ref{t:exp} that we also 
have~$|X|, |Y| \geq C/5 \geq n/250$ asymptotically almost surely. Since no pendant cherry of~$T$ 
can be counted by both~$X$ and~$Y$, it follows that~$T$ is~$\frac{1}{250}$-nice.
\end{proof}

\section{Concluding remarks} \label{s:conclusion}

Recall that Theorem~\ref{t:good-unavoidable} states that all large nice oriented trees of polylogarithmic 
maximum degree are unavoidable. Together with Moon's theorem on the maximum degree of a random labelled 
tree (Theorem~\ref{t:aas-max-degree}) and our proof that almost all labelled oriented trees are 
nice (Theorem~\ref{t:most-are-nice}) this established Theorem~\ref{t:most-unavoidable}, 
that almost all labelled oriented trees are unavoidable.

The same method can be used to show that other classes of random oriented trees
are asymptotically almost surely unavoidable. More precisely, let~$\calt$ be a 
class of undirected trees, let~$\calt_n$ consist of all members of~$\calt$ with~$n$ vertices, 
and let~$T$ be a tree selected uniformly at random from~$\calt_n$. If we can show, for 
some constants~$C$ and~$\xi$, that
\begin{enumerate}[label=(\alph*)]
\item \label{i:treepropa} $\Delta(T) \leq (\log n)^C$ asymptotically almost surely, and
\item \label{i:treepropb} $T$ has at least~$\xi n$ pendant stars asymptotically almost surely,
\end{enumerate}
then by a similar argument to the proof of Theorem~\ref{t:most-are-nice} it follows that a 
uniformly-random orientation~$T^*$ of~$T$ is asymptotically almost surely~$\alpha$-nice 
(where~$\alpha \ll \xi$), and therefore by Theorem~\ref{t:good-unavoidable} that~$T^*$ is 
asymptotically almost surely unavoidable. Following the methods of \Janson~\cite{janson2012simply} 
it is not hard to show that~\ref{i:treepropa} and~\ref{i:treepropb} hold for many classes~$\calt$ of simply-generated random trees, 
such as uniformly-random ordered trees (see~\cite[Example 10.1]{janson2012simply}), 
binary trees (see~\cite[Example 10.3]{janson2012simply}) and~$d$-ary trees 
for a fixed integer~$d \geq 3$ (see~\cite[Example~10.6]{janson2012simply})
In the same way Theorem~\ref{t:good-unavoidable} directly shows that for many fixed trees~$T$, 
such as not-too-unbalanced~$d$-ary trees for a fixed integer~$d \geq 3$, 
a random orientation of~$T$ is unavoidable with high probability. 
Finally we note that for many oriented trees it is straightforward to check directly that 
the conditions of Theorem~\ref{t:good-unavoidable} are satisfied, for instance in the case 
of balanced antidirected binary trees, in which every non-leaf vertex has one child as an 
inneighbour and one child as an outneighbour.

However, there do exist oriented trees which are not nice but which are unavoidable, such as the paths and 
claws discussed in Section~\ref{s:intro}. In this context it is natural to ask whether the property of being 
unavoidable can be succinctly characterised or easily tested.

\begin{question}~
\begin{enumerate}[label=(\roman*)]
\item Is there a concise characterisation of unavoidable oriented trees? 
\item Given an oriented tree~$T$, can we determine in polynomial time if~$T$ is unavoidable?
\end{enumerate}
\end{question}

We suspect that it would be very difficult to establish such a characterisation. As a more 
attainable goal, it would be interesting to establish further classes of unavoidable oriented trees. 
For example, say that an oriented tree~$T$ with root~$r$ is \defi{outbranching} if for every 
vertex~$v \in V(T)$ the path in~$T$ from~$r$ to~$v$ is directed from~$r$ to~$v$. 
In particular, if the root of~$T$ is not a leaf then~$T$ then has no in-leaves at all, 
so~$T$ is not~$\alpha$-nice for any~$\alpha > 0$. 

\begin{problem}
What conditions are sufficient to ensure that an outbranching oriented tree~$T$ is unavoidable?
\end{problem}

To shed some light on this problem it may help to consider the outbranching balanced 
binary trees~$B_d$ on~$2^{d+1}-1$ vertices, in which every non-leaf vertex has two children 
as outneighbours and every leaf is at distance precisely~$d$ from the root. 

\begin{conjecture} \label{conj}
$B_d$ is unavoidable for~$d$ sufficiently large (possibly~$d > 1$ is sufficient).
\end{conjecture}

It seems that further new ideas and techniques would be necessary to prove Conjecture~\ref{conj}, since the 
existence of both many in-leaves and many out-leaves of~$T$ is crucial to the approach we use in this paper.

Finally, recall that in Section~\ref{s:intro} we defined~$g(T)$ for an oriented tree~$T$ to be the 
smallest integer such that every tournament on~$g(T)$ vertices contains a copy of~$T$. 
So~$T$ is unavoidable if and only if~$g(T) = |T|$. As noted earlier, if~$T$ is an out-directed star 
on~$n$ vertices then~$g(T) \geq 2n-2$, and \Kuehn, \Mycroft\ and \Osthus's proof of Sumner's conjecture 
for large trees shows that this is the maximum possible value of~$g(T)$ for large~$n$.
That is, every oriented tree~$T$ on~$n$ vertices, where $n$ is large, has~$g(T) \leq 2n-2$. 
The following `double-star' construction due to Allen and Cooley (see~\cite{KMO11:sumner_approximate}) 
also yields an oriented tree~$T$ for which~$g(T)$ is significantly larger than~$|T|$. 
Fix~$a, b, c \in \bbn$ with~$a+b+c = n$, and let~$T$ be the oriented tree on~$n$ vertices formed from 
a directed path~$P$ on~$b$ vertices by adding~$a$ new vertices as inneighbours of the initial vertex 
of~$P$ and adding~$c$ new vertices as outneighbours of the terminal vertex of~$P$.
Now take disjoint sets of vertices~$A, B$ and~$C$ of sizes~$2a-1, b-1$ and~$2c-1$ respectively, and 
let~$G$ be the tournament in which~$G[A]$ and~$G[C]$ are regular tournaments,~$G[B]$ is an arbitrary 
tournament, and all remaining edges of~$G$ are directed from~$A$ to~$B$, from~$B$ to~$C$ or from~$A$ to~$C$. 
So~$G$ has~$2a+ b+ 2c - 3 = 2n - b - 3$ vertices, but~$G$ does not contain a copy of~$T$, since then 
(as~$|B| < b$) either the initial vertex of~$P$ would be in~$A$, which cannot occur since each vertex 
of~$A$ has only~$a-1$ inneighbours, or the terminal vertex of~$P$ would be in~$C$, which cannot occur 
since each vertex of~$C$ has only~$c-1$ outneighbours. So~$g(T) \geq 2n-b-2$ 
(and it is not too hard to check that in fact~$g(T) = 2n-b-2$).

For any~$\Delta, n \in \bbn$, taking~$a = c = \Delta-1$ and~$b = n - 2\Delta +2$ in the 
above construction yields an oriented tree~$T$ on~$n$ vertices with~$\Delta(T) = \Delta$ 
and~$g(T) = n+2\Delta-4$. In other words, for any~$n\in\bbn$ and any~$\Delta \geq 3$ there exist oriented trees
on $n$ vertices
with maximum degree at most~$\Delta$ which are not unavoidable. On the other hand, 
Theorem~\ref{t:log-bounded-deg} shows that every oriented tree whose maximum degree is at most 
polylogarithmic in~$n$ is contained in every tournament on~$n+\littleo(n)$ vertices. 
\Kuehn, \Mycroft\ and \Osthus~\cite{KMO11:sumner_approximate} asked whether this~$\littleo(n)$ term can 
be replaced by a constant for oriented trees whose maximum degree is at most a constant~$\Delta$, and 
the previous construction shows that a constant of~$2\Delta - 4$ would be best possible. 
More generally it would be interesting to know whether the previous construction is extremal for any bound on~$\Delta(T)$ (as a function of~$n$).

\begin{question}
Is every oriented tree~$T$ on~$n$ vertices contained in every tournament on~$n+2\Delta(T)-4$ vertices?
If not, for which functions $f(n)$ is it true that every oriented tree~$T$ on~$n$ vertices with $\Delta(T) \leq f(n)$ is contained in every tournament on~$n+2f(n)-4$ vertices? 
\end{question}

\addcontentsline{toc}{section}{Bibliography}
\bibliographystyle{simple-url}
\bibliography{minbib}
\edef\MNlastpage{\arabic{page}}
\end{document}